\documentclass[12pt,english]{amsart}

\input xy
\xyoption{all}

\usepackage{comment}
\usepackage[latin1]{inputenc}
\usepackage{amsfonts,amsmath,amssymb,amsthm}
\usepackage{stmaryrd} 
\usepackage{footmisc}

\newcommand{\Zz}{\mathbb{Z}}
\newcommand{\Qq}{\mathbb{Q}}

{\theoremstyle{plain}
\newtheorem{theorem}{Theorem}[section]    % theorem with number

\newtheorem{twisting lemma}[theorem]{Twisting lemma}
\newtheorem{lemma}[theorem]{Lemma}       % lemma with number
\newtheorem{proposition}[theorem]{Proposition}  % lemma with number
}
{\theoremstyle{remark}
\newtheorem{definition}[theorem]{Definition}      % lemma with number
\newtheorem{remark}[theorem]{Remark}   % theorem without number

}

\def\cm{\hbox{\hbox{\rm C}\kern-5pt{\raise 1pt\hbox{$|$}}}}

\def\lhfl#1#2{\smash{\mathop{\hbox to 12mm{\leftarrowfill}}
\limits^{#1}_{#2}}}
\def\rhfl#1#2{\smash{\mathop{\hbox to 12mm{\rightarrowfill}}
\limits^{#1}_{#2}}}

\def\build#1_#2^#3{\mathrel{
\mathop{\kern 0pt#1}\limits_{#2}^{#3}}}

\def\htrait#1#2{\smash{\mathop{\hbox to 12mm{\hrulefill}}
\limits^{#1}_{#2}}}

\def\sxbullet{{\raise 2pt\hbox{\bf .}}}
\numberwithin{equation}{section}

\begin{document}

\title[Non-parametric sets over number fields]{Non-parametric sets of regular realizations over number fields}

\author{Joachim K{\"{o}}nig}

\email{koenig.joach@technion.ac.il}

\address{Department of Mathematics, Technion - Israel Institute of Technology, Haifa 32000, Israel}

\author{Fran\c cois Legrand}

\email{legrandfranc@technion.ac.il}

\address{Department of Mathematics, Technion - Israel Institute of Technology, Haifa 32000, Israel}

\date{\today}

\maketitle

\begin{abstract} 
Given a number field $k$, we show that, for many finite groups $G$, all the Galois extensions of $k$ with Galois group $G$ cannot be obtained by specializing any given finitely many Galois extensions $E/k(T)$ with Galois group $G$ and $E/k$ regular. Our examples include abelian groups, dihedral groups, symmetric groups, general linear groups over finite fields, etc. We also provide a similar conclusion while specializing any given infinitely many Galois extensions $E/k(T)$ with Galois group $G$ and $E/k$ regular of a certain type, under a conjectural ``uniform Faltings' theorem".
\end{abstract}

\section{Introduction}

\subsection{Topic of the paper}

Given a number field $k$, the {\it{Inverse Galois Problem}} (over $k$) asks whether every finite group $G$ occurs as the Galois group of a Galois extension $F/k$. A classical way to obtain such an extension of $k$ with Galois group $G$ consists in introducing an indeterminate $T$ and in producing a Galois extension $E/k(T)$ with Galois group $G$ and $E/k$ {\it{regular}}\footnote{i.e., $E \cap \overline{\Qq}=k$. See \S2.1 for basic terminology. For short, we say that $E/k(T)$ is a ``$k$-regular Galois extension with Galois group $G$" and that $G$ is a ``regular Galois group over $k$" if there exists such an extension.}. Then, by Hilbert's irreducibility theorem, there exist infinitely many $t_0 \in k$ such that the {\it{specialization}} $E_{t_0}/k$ of $E/k(T)$ at $t_0$ has Galois group $G$. Many finite groups have been realized as Galois groups over $k$ by this method; see, e.g., \cite{MM99}.

Then, one can ask whether this geometric approach to solve the Inverse Galois Problem is optimal. This is known as the {\it{Beckmann-Black Problem}}: given a finite group $G$, is every Galois extension $F/k$ with Galois group $G$ a specialization of some $k$-regular Galois extension $E_F/k(T)$ with the same Galois group? Despite its naive appearance, results on the Beckmann-Black Problem are sparse. Namely, the answer is known to be {\it{Yes}} for abelian groups, symmetric groups, alternating groups, and some dihedral groups; see, e.g., \cite[Theorem 2.2]{Deb01b} for more details and references. Moreover, no counter-example is known.

Here, we consider the following generalization of the Beckmann-Black Problem. Given a finite group $G$ and a positive integer $n$, we say that a set $S$ of $k$-regular Galois extensions of $k(T)$ with Galois group $G$ is {\it{$n$-parametric over $k$}} if any given $n$ Galois extensions of $k$ with Galois group $G$ occur as specializations of a same extension $E/k(T) \in S$; see Definition \ref{def1}. With this phrasing, the Beckmann-Black Problem asks whether every finite group $G$ has a 1-parametric set over $k$. The aim of this paper consists in proving that there exist many finite groups $G$ that have no $n$-parametric set over $k$ of a certain type.

\subsection{Finite sets}

Studying the parametricity of a given finite set $S$ of $k$-regular Galois extensions of $k(T)$ with Galois group $G$ has been started by the second author, in the particular case where the set $S$ consists of a {\it{single}} extension $E/k(T)$; see \cite{Leg15, Leg16a}. Following the terminology there, we say for short that the extension $E/k(T)$ is {\it{parametric over $k$}} if every Galois extension of $k$ with Galois group $G$ is a specialization of $E/k(T)$. In this context, the above definition of $n$-parametricity does of course not depend on the integer $n$. 

Strikingly, no finite group is known not to have a parametric extension over $k$ \footnote{Hence, no finite group is known not to have a finite 1-parametric set over $k$.} while only a few are known to have one. For example, recall that only four finite groups (namely, $\{1\}$, $\Zz/2\Zz$, $\Zz/3\Zz$, and $S_3$) are known to have a parametric extension over $\Qq$. This follows from the classical fact that these groups $G$ are the only ones to have a {\it{one parameter generic polynomial over $\Qq$}} (see \cite[page 194]{JLY02}), i.e., a polynomial $P(T,Y) \in \Qq(T)[Y]$ with Galois group $G$ which provides all the Galois extensions with Galois group $G$ of any given field $L$ of characteristic zero (by specializing the indeterminate $T$ properly). Of course, there might exist finite groups besides these four groups with a finite 1-parametric set over $\Qq$ (or even with a parametric extension over $\Qq$). But there is no such example available in the literature! 

Here, we offer a general approach to show that a given finite group $G$ has no finite 1-parametric set over $k$ \footnote{with the possible exception of the empty set; see Remark \ref{rem1}.}. Below, we give a few examples in the specific case $k=\Qq$ which illustrate the wide variety of groups we can cover. See Theorems 5.1-3 for our precise results. 

\begin{theorem} \label{thm intro}
Let $G$ be a non-trivial finite group whose order is neither a prime number nor 4.

\vspace{1mm}

\noindent
{\rm{(1)}} Assume that one of the following conditions holds:

\vspace{0.5mm}

{\rm{(a)}} $G$ has order prime to 6,

\vspace{0.5mm}

{\rm{(b)}} $G$ is abelian, but none of the following groups: $\Zz/6\Zz$, $\Zz/12\Zz$, 

$\Zz/2\Zz \times \Zz/4\Zz$, $\Zz/2\Zz \times \Zz/6\Zz$, $(\Zz/3\Zz)^2$, $(\Zz/2\Zz)^3$, $(\Zz/2\Zz)^4$,

\vspace{0.5mm}

{\rm{(c)}} $G$ is the dihedral group $D_n$ (with $2n$ elements), where $n$ is any 

positive integer which is neither a prime number nor in $\{4,6,8,9,12\}$,

\vspace{0.5mm}

{\rm{(d)}} $G={\rm{GL}}_n(\mathbb{F}_q)$, where $n \geq 2$ is any positive integer and $q \geq 3$ is 

any prime power such that $(n,q) \not=(2,3)$.

\vspace{0.5mm}

\noindent
Then, $G$ has no finite 1-parametric set over $\Qq$.

\vspace{1mm}

\noindent
{\rm{(2)}} Assume that one of the following conditions holds:

\vspace{0.5mm}

{\rm{(a)}} $G$ is abelian,

\vspace{0.5mm}

{\rm{(b)}} $G=S_n$ ($n \geq 6$).

\vspace{0.5mm}

\noindent
Then, $G$ has no parametric extension over $\Qq$.
\end{theorem}

\noindent
All our examples are obtained as applications of several criteria which ensure that a given (finite) set $S$ of $k$-regular Galois extensions of $k(T)$ with Galois group $G$ is not 1-parametric over $k$; see \S3. These criteria can be applied to various finite groups and, for the sake of simplicity, we have only considered ``nice" families of groups here. The interested reader is then invited to give more examples. The proofs of our criteria and our explicit examples involve a wide variety of tools, such as a finiteness result for specializations (see \S2.3), solving embedding problems (see \S4), various results on regular Galois realizations (see, e.g., Lemma 6.7), as well as non-existence results of rational points on twisted (hyper)elliptic curves (see \S8). However, our method requires the existence of a non-trivial proper normal subgroup of $G$ with several properties. In particular, this cannot be used if $G$ is a simple group\footnote{In a recent paper of Neftin and the two authors, a completely different method is used to show that $A_n$ ($n \geq 4$) has no finite 1-parametric set over $k$; see \cite[Corollary 7.3]{KLN17}. We also mention this weaker result (which can be used with simple or non-simple groups): any given non-trivial regular Galois group $G$ over $k$ is the Galois group of a $k$-regular Galois extension of $k(T)$ that is not parametric over $k$; see \cite[Theorem 1.3]{Leg16a} and \cite[Theorem 2.2]{Koe17}.}.

Theorem 1.1 can be compared with a recent result of D\`ebes \cite{Deb16}. It is shown there that there exist many finite groups $G$ with this property: no $k$-regular Galois extension $E/k(T)$ with Galois group $G$  is ``parametric over $k(U)$" (with $U$ a new indeterminate), i.e., all the specializations of the extension $E(U)/k(U)(T)$ cannot provide all the Galois extensions of $k(U)$ with Galois group $G$. As an immediate consequence of \cite[Remark 2.3]{Deb16}, we obtain that this weaker conclusion of D\`ebes holds for every finite group $G$ which is covered by our method.

\subsection{Infinite sets}

Finding finite groups with no infinite 1-parametric set over $k$ is a much more challenging problem as this would disprove the Beckmann-Black Problem over the number field $k$ and as such sets occur more often than in the finite case. Stronger results can even happen: for example, recall that, as the polynomial $Y^N + T_1 Y^{N-1} + \cdots + T_{N-1}Y + T_N \in \Qq(T_1,\dots,T_N)[Y]$ is generic, the set of all $k$-regular Galois extensions of $k(T)$ with Galois group $S_N$ ($N \geq 1$) is $n$-parametric over $k$ for every positive integer $n$; see \cite[Proposition 3.3.10]{JLY02}.

Here, we consider the set of all $k$-regular Galois extensions of $k(T)$ with Galois group $G$ and (at most) $r_0$ {\it{branch points}}, where $r_0$ is any given positive integer. This is motivated by the fact that some of these sets are known to be 1-parametric over $k$. For example,

\noindent
- if $G$ is abelian, then the set of all $k$-regular Galois extensions of $k(T)$ with Galois group $G$ and exactly $r_0$ branch points is $1$-parametric over $k$, unless this set is empty; see \cite{Deb99a},

\noindent
- if $G=S_n$ ($n \geq 1$) and $k=\Qq$, then the above set is $1$-parametric over $k$ as soon as $r_0 \geq 2n-1$; see \cite[Proposition 1.2]{Bec94}.

\noindent
We prove that, for many finite groups $G$, the set of all $k$-regular Galois extensions of $k(T)$ with Galois group $G$ and at most $r_0$ branch points is not $n$-parametric over $k$ for large $n$ (depending on $r_0$). This is shown to hold under a conjectural ``uniform Faltings' theorem", asserting that the number of $k$-rational points on a given smooth curve $X$ defined over $k$ with genus $g$ at least 2 is bounded by a constant $B$ depending only on $g$ and $k$ (and not on $X$). See \S2.3, as well as Theorems 5.4-5, for precise statements.

\subsection{Framework of the paper} 

This paper is organized as follows. In \S2, we recall some background that is used in the sequel. \S3 contains our criteria to show the non-parametricity of a given set of regular realizations of a finite group $G$ over a number field $k$. In \S4, we solve embedding problems which are part of the assumptions of the criteria of \S3. \S5 is devoted to the statements of Theorems 5.1-5, while proofs are given in \S6. As to \S7 and \S8, they are devoted to auxiliary results that are used throughout the paper.

\vspace{3mm}

{\bf{Acknowledgments.}} We wish to thank Danny Neftin for helpful discussions, as well as the anonymous referee for many valuable comments. The first author is supported by the Israel Science Foundation (grant No. 577/15). The second author is supported by the Israel Science Foundation (grants No. 696/13, 40/14, and 577/15).

\section{Basics}

Below, we survey some standard machinery that will be used in the sequel. In \S2.1, we review some standard background on function field extensions. \S2.2 is devoted to basic properties of parametric sets and parametric extensions. Finally, we give finiteness results on the number of specialization points with prescribed specialization in \S2.3.

For this section, let $k$ be a number field, $G$ a finite group, and $T$ an indeterminate over $k$.

\subsection{Background on function field extensions}

Let $E/k(T)$ be a Galois extension with Galois group $G$ and such that $E/k$ is {\it{regular}} (i.e., $E \cap \overline{\Qq}=k$). For short, we say that $E/k(T)$ is a {\it{$k$-regular Galois extension with Galois group $G$.}} In the rest of the paper, by ``genus of the extension $E/k(T)$", we mean the genus of the function field $E$.

We propose the following definition:

\begin{definition} \label{def0}
We define the {\it{minimal genus of $G$ over $k$}}, denoted by $m_{G,k},$ as the smallest integer $g$ such that there exists at least one $k$-regular Galois extension of $k(T)$ that has Galois group $G$ and genus $g$. If $G$ is not a regular Galois group over $k$ \footnote{i.e., if there exists no $k$-regular Galois extension of $k(T)$ with Galois group $G$.}, we set $m_{G,k}=\infty$.
\end{definition}

\subsubsection{Branch points}

Recall that $t_0 \in \mathbb{P}^1(\overline{\Qq})$ is {\it{a branch point of $E/k(T)$}} if the prime ideal $(T-t_0)  \overline{\Qq}[T-t_0]$ \footnote{Replace $T-t_0$ by $1/T$ if $t_0 = \infty$.} ramifies in the integral closure of $\overline{\Qq}[T-t_0]$ in the compositum $E\overline{\Qq}$ of $E$ and $\overline{\Qq}(T)$ (in a fixed algebraic closure of $k(T)$). The extension $E/k(T)$ has only finitely many branch points, denoted by $t_1,\dots,t_r$ \footnote{One has $r=0$ if and only if $G$ is trivial.}.

\subsubsection{Inertia canonical invariant}

For each positive integer $n$, fix a primitive $n$-th root of unity $\zeta_n$. Assume that the system $\{\zeta_n\}_n$ is {\it{coherent}}, i.e., $\zeta_{nm}^n=\zeta_m$ for any positive integers $n$ and $m$.

To $t_i$ can be associated a conjugacy class $C_i$ of $G$, called the {\it{inertia canonical conjugacy class (associated with $t_i$)}}, in the following way. The inertia groups of the prime ideals lying over $(T-t_i) \, \overline{\Qq}[T-t_i]$ in the extension ${E}\overline{\Qq}/\overline{\Qq}(T)$ are cyclic conjugate groups of order equal to the ramification index $e_i$. Furthermore, each of them has a distinguished generator corresponding to the automorphism $(T-t_i)^{1/e_i} \mapsto \zeta_{e_i} (T-t_i)^{1/e_i}$ of $\overline{\Qq}(((T-t_i)^{1/e_i}))$. Then, $C_i$ is the conjugacy class of all the distinguished generators of the inertia groups of the prime ideals lying over $(T-t_i) \, \overline{\Qq}[T-t_i]$ in the extension ${E}\overline{\Qq}/\overline{\Qq}(T)$. The unordered $r$-tuple $({C_1},\dots,{C_r})$ is called {\it{the inertia canonical invariant of ${E}/k(T)$}}.

\subsubsection{Specializations}

Given $t_0 \in \mathbb{P}^1(k) \setminus \{t_1,\dots,t_r\}$, the residue extension of $E/k(T)$ at a prime ideal $\mathcal{P}$ lying over $(T-t_0) k[T-t_0]$ is denoted by ${E}_{t_0}/k$ and called {\it{the specialization of ${E}/k(T)$ at $t_0$}}. This does not depend on the prime $\mathcal{P}$  lying over $(T-t_0) k[T-t_0]$ as the extension ${E}/k(T)$ is Galois. The specialization $E_{t_0}/k$ is Galois with Galois group a subgroup of $G$, namely the decomposition group of ${E}/k(T)$ at $\mathcal{P}$.

\subsection{Parametric sets and parametric extensions}

\begin{definition} \label{def1}
Let $n$ be a positive integer and $S$ a set of $k$-regular Galois extensions of $k(T)$ with Galois group $G$.  We say that $S$ is {\it{an $n$-parametric set over $k$}} if, given $n$ extensions $F_1/k, \dots, F_n/k$ each of which has Galois group $G$, there exists some extension $E/k(T)$ in $S$ such that $F_1/k, \dots, F_n/k$ all are specializations of $E/k(T)$.
\end{definition}

\begin{remark} \label{rem1}
{\rm{(1)}} If $S$ is an $n$-parametric set over $k$ for a given integer $n \geq 1$, then it is $m$-parametric over $k$ for every integer $1 \leq m \leq n$.

\vspace{1mm}

\noindent
{\rm{(2)}} Assume that $S$ consists of only one extension $E/k(T)$. Then, given an integer $n \geq 1$, the set $S=\{E/k(T)\}$ is $n$-parametric over $k$ if and only if it is $1$-parametric over $k$. In this case, we will say for short that {\it{the extension $E/k(T)$ is parametric over $k$.}}

\vspace{1mm}

\noindent
{\rm{(3)}} The set $S$ consisting of no $k$-regular Galois extension of $k(T)$ with group $G$ is 1-parametric over $k$ if and only if $G$ is not a Galois group over $k$ \footnote{i.e., if there exists no Galois extension of $k$ with Galois group $G$.}. In this case, this set is $n$-parametric over $k$ for each $n \geq 1$.
\end{remark}

\subsection{Finiteness of the number of specialization points with prescribed specialization}

Proposition \ref{twisting} below, which is already implicitly used in \cite[\S3.3.5]{Deb99a}, will be used on several occasions.

\begin{proposition} \label{twisting}
Let $E/k(T)$ be a $k$-regular Galois extension with Galois group $G$ and genus at least 2, $H$ a subgroup of $G$, and $F/k$ a Galois extension with Galois group $H$. Then, there exist only finitely many points $t_0 \in \mathbb{P}^1(k)$ such that $F/k=E_{t_0}/k$.
\end{proposition}

\begin{proof}
Denote the genus of $E/k(T)$ by $g$. By the Twisting Lemma 3.2 of \cite{DL12}, there exist a positive integer $n \leq (2 \cdot |G|)^{|G|}$ and $n$ smooth curves $X_1, \dots, X_n$ defined over $k$ that satisfy the following properties:

\vspace{0.5mm}

\noindent
(1) $X_1,\dots,X_n$ all have genus $g$,

\vspace{0.5mm}

\noindent
(2) two different points $t_0$ and $t'_0$ in $\mathbb{P}^1(k)$ such that $E_{t_0}=F=E_{t'_0}$ give rise to two different $k$-rational points in $X_1(k) \cup \cdots \cup X_n(k)$.

\vspace{0.5mm}

\noindent
Assume that there exist infinitely many $t_0 \in \mathbb{P}^1(k)$ such that $F/k=E_{t_0}/k$. Then, by (2), $X_1(k) \cup \cdots \cup X_n(k)$ is infinite, i.e., some $X_{i_0}(k)$ is infinite as well. By Faltings' theorem, the genus of $X_{i_0}$ then is 0 or 1, which cannot happen by (1) and the assumption on $g$.
\end{proof}

Now, recall the following standard conjecture which provides Faltings' theorem for families of curves with the same given genus.

\vspace{2mm}

\noindent
{\bf{Uniformity Conjecture.}} {\it{Let $g \geq 2$ be an integer. Then, there exists a positive integer $B$, which depends only on $g$ and $k$, such that, for every smooth curve $X$ defined over $k$ with genus $g$, the set $X(k)$ of all $k$-rationals points on $X$ has cardinality at most $B$.}}

\vspace{2mm}

\noindent
By \cite[Theorem 1.1]{CHM97}, this conjecture is true under the Lang Conjecture, which asserts that the set of all $k$-rational points on a given variety of general type defined over $k$ is not Zariski dense.

\vspace{2mm}

Finally, combining the Twisting Lemma and the Uniformity Conjecture provides the following conjectural version of Proposition \ref{twisting}. As the proof is almost identical, details are left to the reader.

\begin{proposition} \label{twisting2}
Assume that the Uniformity Conjecture holds. Then, given an integer $g_0 \geq 2$,  there exists an integer $B \geq 1$, depending only on $|G|$, $k$, and $g_0$, and satisfying the following. Let $E/k(T)$ be a $k$-regular Galois extension with Galois group $G$ and genus between 2 and $g_0$, $H$ a subgroup of $G$, and $F/k$ a Galois extension with Galois group $H$. Then, there exist at most $B$ points $t_0 \in \mathbb{P}^1(k)$ such that $F/k=E_{t_0}/k$.
\end{proposition}

\section{Criteria for non-parametricity}

For this section, let $k$ be a number field, $O_k$ the integral closure of $\Zz$ in $k$, and $G$ a finite group. 

The aim of this section is to give criteria for the group $G$ to have no parametric set over $k$ with suitable properties.

\subsection{A preliminary result}

Below, we explain how to derive parametric sets with various Galois groups, assuming we know at least one.

\begin{proposition} \label{prelim1}
Let $H$ be a normal sugbroup of $G$. Assume that every Galois extension $F/k$ with Galois group $G/H$ embeds into a Galois extension of $k$ with Galois group $G$ \footnote{i.e., there exists a Galois extension $L/k$ with Galois group $G$ such that $F \subset L$.}. Then, the following two conclusions hold.

\vspace{1mm}

\noindent
{\rm{(1)}} Every $1$-parametric set over $k$ for the group $G$ with positive cardinality at most $s$ gives rise to a $1$-parametric set over $k$ for the group $G/H$ with positive cardinality at most $s \cdot r$, where $r$ denotes the number of normal sugbroups $H'$ of $G$ such that $G/H' \cong G/H$.

\vspace{1mm}

\noindent
{\rm{(2)}} Furthermore, assume that $G$ has a unique normal subgroup $H'$ such that $G/H' \cong G/H$. Let $n$ be a positive integer. Then, every $n$-parametric set over $k$ for the group $G$ with positive cardinality at most $s$ gives rise to an $n$-parametric set over $k$ for the group $G/H$ with positive cardinality at most $s$. In particular, every parametric extension over $k$ with Galois group $G$ gives rise to a parametric extension over $k$ with Galois group $G/H$.
\end{proposition}

We need the following lemma.

\begin{lemma} \label{subextensions}
Let $H$ be a normal subgroup of $G$, $E/k(T)$ a $k$-regular Galois extension with Galois group $G$, and $t_0 \in \mathbb{P}^1(k)$, which is not a branch point of $E/k(T)$. Denote the distinct subextensions of $E/k(T)$ that have Galois group $G/H$ by $E_1/k(T), \dots, E_r/k(T)$. Assume that the specialization $E_{t_0}/k$ has Galois group $G$. Then, the extensions $(E_1)_{t_0}/k, \dots, (E_r)_{t_0}/k$ are the distinct subextensions of $E_{t_0}/k$ that have Galois group $G/H$.
\end{lemma}

\begin{proof}
Let $i$ between 1 and $r$. As $E_{t_0}/k$ has Galois group $G$, the specialized extension $E_{t_0}/(E_i)_{t_0}$ has Galois group ${\rm{Gal}}(E/E_i)$. As $E_1, \dots, E_r$ are distinct, the same is true for ${\rm{Gal}}(E/E_1), \dots, {\rm{Gal}}(E/E_r)$. Then, the extensions $(E_1)_{t_0}/k, \dots, (E_r)_{t_0}/k$ are distinct and they all have Galois group $G/H$. It then remains to notice that $r$ is the number of normal subgroups $H'$ of $G$ such that $G/H' \cong G/H$ to conclude.
\end{proof}

\begin{proof}[Proof of Proposition \ref{prelim1}]
Given a positive integer $n$, let $E_1/k(T), \dots,$ $E_s/k(T)$ be $k$-regular Galois extensions with Galois group $G$ giving rise to an $n$-parametric set over $k$ ($s \geq 1$). Denote the $r$ normal subgroups $H'$ of $G$ such that $G/H' \cong G/H$ by $H_1, \dots, H_r$. Let $F_1/k, \dots, F_n/k$ be $n$ Galois extensions each of which has Galois group $G/H$. By our embedding assumption, there exist $n$ Galois extensions $L_1/k, \dots, L_n/k$ each of which has Galois group $G$ and such that $F_j \subset L_j$ for each $j \in \{1,\dots,n\}$. As the set $\{E_1/k(T), \dots, E_s/k(T)\}$ is $n$-parametric over $k$, there exist $i \in \{1,\dots,s\}$ and specialization points $t_{0,1}, \dots, t_{0,n} \in \mathbb{P}^1(k)$ such that $L_j/k$ occurs as the specialization of $E_i/k(T)$ at $t_{0,j}$ for each $j \in \{1,\dots,n\}$. Let $j \in \{1,\dots,n\}$. By Lemma \ref{subextensions}, the extensions $(E_i^{H_1})_{t_{0,j}}/k, \dots, (E_i^{H_r})_{t_{0,j}}/k$ are the distinct subextensions of $L_j/k$ with Galois group $G/H$. Hence, $F_j/k = (E_i^{H_l})_{t_{0,j}}/k$ for some $l \in \{1,\dots,r\}$. First, assume that $n=1$. Then, $F_1/k = (E_i^{H_l})_{t_{0,1}}/k$ for some $l \in \{1,\dots,r\}$. In particular, the set $$\{E_1^{H_1}/k(T), \dots, E_1^{H_r}/k(T), \dots, E_s^{H_1}/k(T), \dots, E_s^{H_r}/k(T)\}$$ is $1$-parametric over $k$, as needed for part (1). Now, suppose that $r=1$. Then, $F_j/k = (E_i^{H})_{t_{0,j}}/k$ ($j \in \{1,\dots,n\}$). Hence, $\{E_1^{H}/k(T), \dots,$ $E_s^{H}/k(T)\}$ is $n$-parametric over $k$, as needed for part (2).
\end{proof}

\subsection{A criterion in minimal genus $\geq 2$}

Theorem \ref{genus2} below is the easiest and the most useful of our criteria.

\begin{theorem} \label{genus2}
Assume that $G$ has a normal subgroup $H$ such that the following two conditions hold:

\vspace{0.5mm}

\noindent
{\rm{(1)}} there exists at least one Galois extension of $k$ with Galois group $G/H$ which embeds into infinitely many Galois extensions of $k$ with Galois group $G$,

\vspace{0.5mm}

\noindent
{\rm{(2)}} $m_{{G/H},k} \geq 2$.

\vspace{0.5mm}

\noindent
Then, the following two conclusions hold.

\vspace{0.5mm}

\noindent
{\rm{(1)}} The group $G$ has no finite 1-parametric set over $k$. 

\vspace{0.5mm}

\noindent
{\rm{(2)}} Furthermore, assume that the Uniformity Conjecture of \S2.3 holds. Then, given an integer $r_0 \geq 1$, there exists an integer $n \geq 1$ (depending only on $r_0$, $k$, and $|G|$) such that the set which consists of all $k$-regular Galois extensions of $k(T)$ with Galois group $G$ and at most $r_0$ branch points is not $n$-parametric over $k$.
\end{theorem}

\noindent
In the light of condition (2) above, it is useful to investigate which finite groups can occur as the Galois group of a $k$-regular Galois extension of $k(T)$ with genus $\leq 1$. We refer to \S7 for more on this classical topic. As to the embedding condition (1), it is studied in \S4.

Theorem \ref{genus2} is a straightforward application of Lemma \ref{lemma genus2} below.

\begin{lemma} \label{lemma genus2}
Let $r_0$ be a positive integer and $S$ a set of $k$-regular Galois extensions of $k(T)$ with Galois group $G$. Assume that $G$ has a normal subgroup $H$ such that the following two conditions hold:

\vspace{0.5mm}

\noindent
{\rm{(1)}} there exists a Galois extension of $k$ with group $G/H$ which embeds into infinitely many Galois extensions of $k$ with group $G$,

\vspace{0.5mm}

\noindent
{\rm{(2)}} for each extension $E/k(T) \in S$ and each $j \in \{1,\dots,r\}$, the genus of $E^{H_j}/k(T)$ is at least 2, where $H_1, \dots, H_r$ denote the distinct normal subgroups $H'$ of $G$ such that $G/H' \cong G/H$.

\vspace{0.5mm}

\noindent
Then, the following two conclusions hold.

\vspace{0.5mm}

\noindent
{\rm{(1)}} Assume that the set $S$ is finite. Then, there exist infinitely many Galois extensions $F/k$ with Galois group $G$ such that, for each $E/k(T) \in S$, the extension $F/k$ is not a specialization of $E/k(T)$.

\vspace{0.5mm}

\noindent
{\rm{(2)}} Assume that each $E/k(T)$ in $S$ has at most $r_0$ branch points and the Uniformity Conjecture holds. Then, there exist an integer $n$ (depending only on $r_0$, $k$, and $|G|$) and $n$ Galois extensions $L_1/k, \dots, L_n/k$ with Galois group $G$ such that, for each $E/k(T) \in S$, at least one of the extensions $L_1/k, \dots, L_n/k$ is not a specialization of $E/k(T)$.
\end{lemma}

\begin{proof}
First, assume that $S$ is finite. As condition (1) holds, conclusion (1) holds if $S$ is empty. We may then assume that $S$ is not empty. Set $S= \{E_1/k(T), \dots, E_s/k(T)\}$ with $s \geq 1$. Assume that all but finitely many Galois extensions of $k$ with Galois group $G$ are specializations of $E_1/k(T), \dots, E_s/k(T)$. Let $F/k$ be a Galois extension with Galois group $G/H$ which embeds into infinitely many Galois extensions $L/k$ with Galois group $G$ (condition (1)). Let $L/k$ be such an extension. Up to dropping finitely many of them, we may assume that $L/k = (E_i)_{t_0}/k$ for some $i \in \{1,\dots,s\}$ and some $t_0 \in \mathbb{P}^1(k)$. By Lemma \ref{subextensions}, $(E_i^{H_1})_{t_0}/k, \dots, (E_i^{H_r})_{t_0}/k$ are the distinct subextensions of $L/k$ that have Galois group $G/H$. Hence, $F/k=(E_i^{H_j})_{t_0}/k$ for some $j \in \{1,\dots,r\}$. Repeat this trick infinitely many times to get that there exists $(i,j) \in \{1,\dots,s\} \times \{1,\dots,r\}$ and infinitely many $t_0 \in \mathbb{P}^1(k)$ such that $(E_i^{H_j})_{t_0}=F$. By Proposition \ref{twisting}, the genus of $E_i^{H_j}/k(T)$ is 0 or 1, which cannot happen by condition (2). Hence, conclusion (1) holds.

Now, assume that each extension $E/k(T) \in S$ has at most $r_0$ branch points and the Uniformity Conjecture holds. Let $F/k$ be a Galois extension with Galois group $G/H$ and a sequence $(L_n/k)_{n \geq1}$ of distinct Galois extensions with Galois group $G$ containing $F$. Given $n \geq 1$, assume that $L_1/k, \dots, L_n/k$ all are specializations of some extension $E/k(T)$ in $S$. Apply the same argument as before to get that there exists $j \in \{1,\dots,r\}$ and $n/r$ distinct points $t_0 \in \mathbb{P}^1(k)$ such that $F/k=(E^{H_j})_{t_0}/k$. As $E/k(T)$ has at most $r_0$ branch points, the same is true of the subextension $E^{H_j}/k(T)$. Hence, the genus of $E^{H_j}/k(T)$, which is at least 2 (condition (2)), is bounded by some integer $g_0$ which depends only on $|G|$ and $r_0$. One may then apply Proposition \ref{twisting2} to get that there exists a positive integer $B$, which depends only on $|G|$, $r_0$, and $k$, such that the set of points $t_0 \in \mathbb{P}^1(k)$ satisfying $(E^{H_j})_{t_0}/k=F/k$ has cardinality at most $B$. Hence, $n \leq r \cdot B$, thus ending the proof.
\end{proof}

\subsection{A criterion in minimal genus $\geq 1$}

Below, we give an analog of Theorem \ref{genus2}, where we relax condition (2) on the minimal genus. However, we have to add some local conditions in the embedding condition (1) and drop the conjectural conclusion (2).

\begin{theorem} \label{genus1}
Assume that $G$ has a normal subgroup $H$ such that the following two conditions hold.

\vspace{0.5mm}

\noindent
{\rm{(1)}} There exists an infinite set $\mathcal{S}$ of prime ideals of $O_k$ satisfying the following. Given $\mathcal{P} \in \mathcal{S}$, there exists a Galois extension of $k$ with group $G/H$, with ramification index at $\mathcal{P}$ not in $\{1,2,3,4,6\}$, and which embeds into infinitely many Galois extensions of $k$ with group $G$.

\vspace{0.5mm}

\noindent
{\rm{(2)}} One has $m_{{G/H},k} \geq 1$.

\vspace{0.5mm}

\noindent
Then, $G$ has no finite 1-parametric set over $k$. 
\end{theorem}

\noindent
{\it{Addendum}} \ref{genus1}. In the special case $k=\Qq$, one may replace condition (1) by the following weaker variant.

\vspace{1mm}

\noindent
(1)' There exists an infinite set $\mathcal{S}$ of prime numbers satisfying the following. Given $p \in \mathcal{S}$, there exists a Galois extension of $\Qq$ with Galois group $G/H$, with ramification index at $p$ at least 3, and which embeds into infinitely many Galois extensions of $\Qq$ with Galois group $G$.

\vspace{2mm}

Theorem \ref{genus1} is a straightforward application of Lemma \ref{lemma genus1} below.

\begin{lemma} \label{lemma genus1}
Let $S$ be a finite set of $k$-regular Galois extensions of $k(T)$ with Galois group $G$. Assume that $G$ has a normal subgroup $H$ such that the following two conditions hold.

\vspace{0.5mm}

\noindent
{\rm{(1)}} There exists an infinite set $\mathcal{S}$ of prime ideals of $O_k$ satisfying the following. Given $\mathcal{P} \in \mathcal{S}$, there exists a Galois extension of $k$ with group $G/H$, with ramification index at $\mathcal{P}$ not in $\{1,2,3,4,6\}$, and which embeds into infinitely many Galois extensions of $k$ with group $G$.

\vspace{0.5mm}

\noindent
{\rm{(2)}} For each extension $E/k(T) \in S$ and each $j \in \{1,\dots,r\}$, the genus of $E^{H_j}/k(T)$ is at least 1, where $H_1, \dots, H_r$ denote the distinct normal subgroups $H'$ of $G$ such that $G/H' \cong G/H$.

\vspace{0.5mm}

\noindent
Then, there exist infinitely many Galois extensions $F/k$ with Galois group $G$ such that, for each extension $E/k(T) \in S$, the extension $F/k$ is not a specialization of $E/k(T)$. Moreover, in the special case $k=\Qq$, one may replace condition {\rm{(1)}} by condition {\rm{(1)}}' of Addendum \ref{genus1}.
\end{lemma}

\begin{proof}
As in the case $m_{G/H,k} \geq 2$, one may assume $S \not= \emptyset$. Set $S= \{E_1/k(T), \dots, E_s/k(T)\}$ with $s \geq 1$. Assume that all but finitely many Galois extensions of $k$ with group $G$ are specializations of $E_1/k(T), \dots, $ $E_s/k(T)$. Let $\mathcal{P}$ be one of the infinitely many primes of $\mathcal{S}$ of condition (1) which is a {\it{bad prime}}\footnote{See \cite[Definition 2.6]{Leg16c} for the precise definition. Here, we only use the standard fact that a given finite $k$-regular Galois extension of $k(T)$ has only finitely many bad primes.} for none of the subextensions $E_1^{H_1}/k(T), \dots,$ $E_1^{H_r}/k(T), \dots, E_s^{H_1}/k(T), \dots, E_s^{H_r}/k(T).$ Let $F/k$ be a Galois extension with group $G/H$, ramification index at $\mathcal{P}$ not in $\{1,2,3,4,6\}$, and which embeds into infinitely many Galois extensions $L/k$ with group $G$ (condition (1)). As in the case $m_{G/H,k} \geq 2$, one finds $(i,j) \in \{1,\dots,s\} \times \{1,\dots,r\}$ and infinitely many $t_0 \in \mathbb{P}^1(k)$ such that $(E_i^{H_j})_{t_0}=F$. Then, by Proposition \ref{twisting} and condition (2), the genus of $E_i^{H_j}/k(T)$ is 1. Hence, by Proposition \ref{RH genus 1}, the inertia canonical invariant of this extension consists only of conjugacy classes of elements of order 2, 3, 4 or 6. As $\mathcal{P}$ is a good prime for $E_i^{H_j}/k(T)$, we may apply the Specialization Inertia Theorem \cite[\S2.2.3]{Leg16c}\footnote{The assumption requiring the branch points of $E_i^{H_j}/k(T)$ and their inverses to be integral with respect to $\mathcal{P}$ can also be satisfied (up to dropping more primes).} to get that the ramification index of $\mathcal{P}$ in $F/k$ is 1, 2, 3, 4 or 6, a contradiction.

To get the conclusion under condition (1)' in the case $k=\Qq$, it suffices to see that, with the same notation as above and by Proposition 7.2, the inertia canonical invariant of $E_i^{H_j}/\Qq(T)$ should consist only of conjugacy classes of involutions of $G/H$.
\end{proof}

\subsection{Criteria with no explicit assumption on the minimal genus}

In Theorem \ref{genus0} below, we do not require any assumption on the minimal genus of quotients of $G$, at the cost of 
making the conclusion weaker.

\begin{theorem} \label{genus0}
Assume that $G$ has a normal subgroup $H$ such that the following two conditions hold.

\vspace{0.5mm}

\noindent
{\rm{(1)}} There exists a Galois extension of $k$ with group $G/H$ which embeds into infinitely many Galois extensions of $k$ with  group $G$.

\vspace{0.5mm}

\noindent
{\rm{(2)}} For each normal subgroup $H'$ of $G$ such that $G/H' \cong G/H$, there exist 5 conjugacy classes $C_1, \dots, C_5$ of $G$ such that the following holds for each $i \in \{1,\dots,5\}$:

\vspace{0.3mm}

{\rm{(a)}} $C_i$ is not contained in $H'$,

\vspace{0.3mm}

{\rm{(b)}} each element of $C_i$ generates a maximal cyclic subgroup of $G$ \footnote{By this, we mean a cyclic subgroup of $G$ which is contained in no strictly larger cyclic subgroup of $G$.}, 

\vspace{0.3mm}

{\rm{(c)}} for each $j \in \{1,\dots,5\} \setminus \{i\}$, the class $C_i$ is not a power of $C_j$,

\vspace{0.3mm}

{\rm{(d)}} $C_i$ belongs to the inertia canonical invariant of some $k$-regular 

Galois extension of $k(T)$ with Galois group $G$.

\vspace{0.5mm}

\noindent
Then, given a $k$-regular Galois extension $E/k(T)$ with Galois group $G$, there exist infinitely many Galois extensions of $k$ with Galois group $G$ each of which is not a specialization of $E/k(T)$. In particular, $G$ has no parametric extension over $k$.
\end{theorem}

\begin{proof}
Let $E/k(T)$ be a $k$-regular Galois extension with Galois group $G$. Assume that all but finitely many Galois extensions of $k$ with Galois group $G$ are specializations of $E/k(T)$. Let $H'$ be a normal subgroup of $G$ such that $G/H' \cong G/H$. Below, we show that the subextension $E^{H'}/k(T)$ has at least 5 branch points. Then, by the Riemann-Hurwitz formula, the genus of $E^{H'}/k(T)$ is at least 2. We may then combine Lemma \ref{lemma genus2} and condition (1) to get a contradiction.

Let $i \in \{1,\dots,5\}$. By part (d) of condition (2) and our assumption on $E/k(T)$, there exists a conjugacy class $C'_i$ of $G$ belonging to the inertia canonical invariant of $E/k(T)$ and an integer $a_i \geq 1$ such that ${C'_{i}}^{a_i}=C_i$; see \cite[Theorem 4.2]{Leg16c}. Then, by part (b) of condition (2), one has $C'_i=C_i^{b_i}$ for some integer $b_i \geq 1$ which is coprime to the order $c_i$ of the elements of $C_i$. Moreover, by part (c) of condition (2), one has $C'_i \not=C'_j$ for $i \not=j$. Hence, $E/k(T)$ has at least 5 distinct branch points $t_1, \dots, t_5$ whose associated inertia canonical conjugacy classes are $C_1^{b_1}, \dots, C_5^{b_5}$. By part (a) of condition (2) and as $b_i$ is coprime to $c_i$, one has that, for each $i \in \{1,\dots,5\}$, the class $C_i^{b_i}$ is not contained in $H'$. Hence, $t_1,\dots,t_5$ all are branch points of $E^{H'}/k(T)$, as needed.
\end{proof}

Viewing Theorems \ref{genus2} and \ref{genus1}, it is natural to apply Theorem \ref{genus0} when $m_{G/H,k}=0$, i.e., when $G/H$ is one of the groups given in Proposition \ref{list genus 0}. We focus below on the case where $G/H$ is cyclic and leave the other cases to the reader. For simplicity, we take $k=\Qq$.

\begin{theorem} \label{cyclic}
Assume that $G$ has a normal subgroup $H$ such that the following three conditions hold.

\vspace{0.5mm}

\noindent
{\rm{(1)}} The quotient $G/H$ is isomorphic to $\Zz/n\Zz$ for some integer $n \geq 3$.

\vspace{0.5mm}

\noindent
{\rm{(2)}} There exists a Galois extension of $\Qq$ with Galois group $\Zz/n\Zz$ which embeds into infinitely many Galois extensions of $\Qq$ with group $G$.

\vspace{0.5mm}

\noindent
{\rm{(3)}} For each normal subgroup $H'$ of $G$ such that $G/H' \cong \Zz/n\Zz$, there exist 3 distinct conjugacy classes $C_1$, $C_2$, and $C_3$ of $G$ such that the following holds for each $i \in \{1,2,3\}$:

\vspace{0.5mm}

{\rm{(a)}} $C_i$ is not contained in $H'$,

\vspace{0.5mm}

{\rm{(b)}} each element of $C_i$ generates a maximal cylic subgroup of $G$, 

\vspace{0.5mm}

{\rm{(c)}} $C_i$ belongs to the inertia canonical invariant of some $\Qq$-regular 

Galois extension of $\Qq(T)$ with Galois group $G$.

\vspace{0.5mm}

\noindent
Then, given a $\Qq$-regular Galois extension $E/\Qq(T)$ with Galois group $G$, there exist infinitely many Galois extensions of $\Qq$ with Galois group $G$ each of which is not a specialization of $E/\Qq(T)$. In particular, $G$ has no parametric extension over $\Qq$.
\end{theorem}

\begin{proof}
The proof is almost identical to that of Theorem \ref{genus0}. To avoid confusion, we reproduce it below with the needed adjustments.

Let $E/\Qq(T)$ be a $\Qq$-regular Galois extension with Galois group $G$. Assume that all but finitely many Galois extensions of $\Qq$ with Galois group $G$ are specializations of $E/\Qq(T)$. Let $H'$ be a normal subgroup of $G$ such that $G/H' \cong \Zz/n\Zz$. Below, we show that $E^{H'}/\Qq(T)$ has at least 3 branch points. Then, by Propositions 7.1-2 and \ref{list genus leq 1 abelian} (and as $n \geq 3$), the genus of $E^{H'}/\Qq(T)$ is at least 2. We may then combine Lemma \ref{lemma genus2} and condition (2) to get a contradiction.

Let $i \in \{1,2,3\}$. By part (c) of condition (3) and our assumption on $E/\Qq(T)$, there exists a conjugacy class $C'_i$ of $G$ in the inertia canonical invariant of $E/\Qq(T)$ and an integer $a_i \geq 1$ such that ${C'_{i}}^{a_i}=C_i$; see \cite[Theorem 4.2]{Leg16c}. Then, by part (b) of condition (3), one has $C'_i=C_i^{b_i}$ for some integer $b_i \geq 1$ which is coprime to the order of the elements of $C_i$. By the Branch Cycle Lemma (see \cite{Fri77} and \cite[Lemma 2.8]{Vol96}), $C_i$ is in the inertia canonical invariant of $E/\Qq(T)$. As $C_1, \dots, C_3$ are distinct and as none of them is contained in $H'$ (part (a) of condition (3)), $E^{H'}/\Qq(T)$ has at least 3 branch points, as needed.
\end{proof}

The case where $G/H \cong \Zz/2\Zz$ requires more attention as this group has a regular realization over $\Qq$ with genus 1 and 4 branch points.

\begin{theorem} \label{cyclic2}
Assume that the following four conditions hold. 

\vspace{0.5mm}

\noindent
{\rm{(1)}} The group $G$ has a unique index 2 subgroup $H$.

\vspace{0.5mm}

\noindent
{\rm{(2)}} Each quadratic extension of $\Qq$ embeds into a Galois extension of $\Qq$ with Galois group $G$.

\vspace{0.5mm}

\noindent
{\rm{(3)}} There exists at least one quadratic extension of $\Qq$ which embeds into infinitely many Galois extensions of $\Qq$ with Galois group $G$.

\vspace{0.5mm}

\noindent
{\rm{(4)}} There exist $s$ distinct conjugacy classes $C_1, \dots,C_s$ of $G$ such that

\vspace{0.5mm}

{\rm{(a)}} either $s \geq 3$ or $s=2$ and $C_2$ is a power of $C_1$,

\vspace{0.5mm}

{\rm{(b)}} for each $i \in \{1,\dots,s\}$, the class $C_i$ is not contained in $H$,

\vspace{0.5mm}

{\rm{(c)}} for each $i \in \{1,\dots,s\}$, each element of $C_i$ generates a maximal 

cyclic subgroup of $G$, 

\vspace{0.5mm}

{\rm{(d)}} for each $i \in \{1,\dots,s\}$, $C_i$ belongs to the inertia canonical inva-

riant of some $\Qq$-regular Galois extension of $\Qq(T)$ with group $G$.

\vspace{0.5mm}

\noindent
Then, given a $\Qq$-regular Galois extension $E/\Qq(T)$ with group $G$, there exist infinitely many Galois extensions of $\Qq$ with  group $G$ which are not specializations of $E/\Qq(T)$. {\hbox{Hence, $G$ has no parametric extension over $\Qq$.}}
\end{theorem}

\begin{proof}
Let $E/\Qq(T)$ be a $\Qq$-regular Galois extension with Galois group $G$. Assume that all but finitely many Galois extensions of $\Qq$ with Galois group $G$ are specializations of $E/\Qq(T)$. As $H$ is unique (condition (1)), it suffices, as in the proof of Theorem \ref{genus0}, to show that $E^H/\Qq(T)$ has at least 5 branch points. One shows as in the proof of Theorem \ref{cyclic} that $E^H/\Qq(T)$ has at least $s$ branch points whose inertia canonical conjugacy classes are $C_1,\dots,C_s$. Moreover, $H$ is unique and condition (2) holds. Then, one shows as in the proof of Proposition \ref{prelim1} that all but finitely many quadratic extensions of $\Qq$ are specializations of $E^H/\Qq(T)$. Assume that $E^H/\Qq(T)$ has at most 4 branch points. Then, by Proposition \ref{r leq 4}, $E^H/\Qq(T)$ has exactly 2 branch points and both are $\Qq$-rational. As this cannot happen if $s \geq 3$, one has $s=2$ and $C_2$ is a power of $C_1$. Then, by the Branch Cycle Lemma, $E^H/\Qq(T)$ has at least 2 branch points that are not $\Qq$-rational, which cannot happen.
\end{proof}

\section{Solving embedding problems}

In this section, we give sufficient conditions for the embedding conditions of the criteria of \S3 to hold. From now on, we fix a number field $k$, a finite group $G$, and a non-trivial normal subgroup $H$ of $G$.

\subsection{Solving at least one embedding problem infinitely many times}

First, we give sufficient conditions for the following to hold:

\vspace{1mm}

\noindent
{\rm{($*$)}} {\it{there exists a Galois extension of $k$ with group $G/H$ which embeds into infinitely many Galois extensions of $k$ with group $G$.}}

\vspace{1mm}

\noindent
Obviously, a necessary condition for condition {\rm{($*$)}} to hold is that $G$ occurs as a Galois group over $k$.

\subsubsection{The case where $H$ is solvable}

In Proposition \ref{solvable kernel} below, we show that the converse holds if $H$ is solvable.

\begin{proposition} \label{solvable kernel}
{\rm{(1)}} Assume that $H$ is solvable. Then, condition {\rm{($*$)}} holds if and only if $G$ is a Galois group over $k$.

\vspace{0.5mm}

\noindent
{\rm{(2)}} Assume that $H$ is abelian and $H$ is the unique normal subgroup $H'$ of $G$ such that $G/H' \cong G/H$. Then, any given Galois extension of $k$ with Galois group $G/H$ embeds into zero or infinitely many Galois extensions of $k$ with Galois group $G$.
\end{proposition}

\begin{remark} \label{odd order}
By a classical result of Shafarevich (see, e.g., \cite[Theorem (9.6.1)]{NSW08}), solvable groups are Galois groups over $k$. Then, condition {\rm{($*$)}} holds if $G$ is solvable (in particular, if $G$ is abelian or of odd order; see \cite{FT63}) and $H$ is any non-trivial normal subgroup of $G$.
\end{remark}

\begin{proof}
First, we prove (1). Up to replacing the solvable group $H$ by the smallest non-trivial term of its derived series (which is an abelian characteristic subgroup of $H$), we may and will assume that $H$ is abelian. Let $\varphi$ be the canonical surjection of $G$ onto $G/H$ and $n$ a positive integer. Consider the finite group
$$G_{\varphi}^n = \{(g_1, \dots, g_n) \in G^n \, : \, \varphi(g_1) = \cdots = \varphi(g_n)\}.$$
Set $$N=\{(g_1, \dots, g_n) \in G_{\varphi}^n \, : \, g_1 \in H, \dots, g_n \in H\}.$$
Since $H$ is a normal subgroup of $G$, $N$ is a normal subgroup of $G_{\varphi}^n$. Moreover, for each $i \in \{1,\dots,n\}$, set
$$N_i=\{(g_1, \dots, g_n) \in N \, : \, g_i=1\}.$$ Then, $N_1, \dots, N_n$ all are normal subgroups of $G_{\varphi}^n$ contained in $N$.

\begin{lemma}  \label{iso}
{\rm{(1)}} One has $G_{\varphi}^n / N_i \cong G$ for each $i \in \{1,\dots,n\}$.

\vspace{0.5mm}

\noindent
{\rm{(2)}} One has $G_{\varphi}^n / N \cong G/H$.
\end{lemma}

\begin{proof}
To prove (1), it suffices to notice that, given $i \in \{1,\dots,n\}$, $(g_1, \dots,g_n) \in G^n_ {\varphi} \mapsto g_i \in G$ is an epimorphism with kernel $N_i$. As to (2), note that the surjection $\varphi$ induces an epimorphism $(g_1, \dots,g_n) \in G^n_ {\varphi} \mapsto \varphi(g_1) \in G/H$ with kernel $N$. Hence, the lemma holds.
\end{proof}

Let $L/k$ be a Galois extension with Galois group $G$ and $\theta : G \rightarrow {\rm{Gal}}(L/k)$ an isomorphism. Consider the finite {\it{embedding problem}}\footnote{We refer to \cite[Definition 16.4.1]{FJ08} for more on the terminology.}
$$\alpha : (g_1,\dots,g_n) \in G_{\varphi}^n \mapsto \theta (g_1) \in {\rm{Gal}}(L/k).$$
Then, $\alpha$ is split and, by the proof of Lemma \ref{iso}, it has kernel $N_1$, which is abelian (as $H$ is). By a result of Ikeda (see, e.g., \cite[Proposition 16.4.5]{FJ08}), every finite split embedding problem with abelian kernel over a Hilbertian field is solvable\footnote{More generally, Proposition \ref{solvable kernel} holds if $k$ is Hilbertian.}. Thus, there exist a Galois extension $M/k$ such that $L \subset M$ and an isomorphism $\beta : {\rm{Gal}}(M/k) \rightarrow G^n_\varphi$ such that $\alpha \circ \beta$ maps every element of ${\rm{Gal}}(M/k)$ to its restriction to $L$. 

Consider the Galois subextensions $M^{\beta^{-1}(N_1)}/k, \dots, M^{\beta^{-1}(N_n)}/k.$
They satisfy the following three properties:

\noindent
- ${\rm{Gal}}(M^{\beta^{-1}(N_i)}/k) \cong G$ for each $i \in \{1,\dots,n\}$ (Lemma \ref{iso}),

\noindent
- $M^{\beta^{-1}(N_1)}/k = L/k$ (this follows from $\alpha \circ \beta$ mapping every element of ${\rm{Gal}}(M/k)$ to its restriction to $L$),

\noindent
- $M^{\beta^{-1}(N_i)}/k \not= M^{\beta^{-1}(N_j)}/k$ for $i \not=j$ (as $H$ is not trivial).

\noindent
Moreover, all of them contain the Galois extension $M^{\beta^{-1}(N)}/k$ which has Galois group isomorphic to $G/H$ (Lemma \ref{iso}). Hence, there exists a subextension of $L/k$ with Galois group $G/H$ which embeds into $n$ distinct Galois extensions of $k$ with Galois group $G$. As the given extension $L/k$ has only finitely many subextensions, we are done.

Now, we prove (2). Assume that $H$ is abelian and unique. With no loss of generality, we may assume that $G/H$ is a Galois group over $k$. Fix a Galois extension $F/k$ with Galois group $G/H$ which embeds into at least one Galois extension $L/k$ with Galois group $G$. Apply the same argument as above with the extension $L/k$ to get that, for every integer $n \geq 1$, some subextension $F_n/k$ of $L/k$ with Galois group $G/H$ embeds into $n$ distinct Galois extensions of $k$ with Galois group $G$. As $H$ is unique, one has $F_n=F$ for each $n \geq 1$, thus ending the proof.
\end{proof}

\subsubsection{The case where $H$ is not solvable.}

To our knowledge, there is no general criterion asserting that condition ($*$) holds if and only if $G$ is a Galois group over $k$ if $H$ is not solvable. However, this is known for some specific non-abelian simple groups $H$. Here are some of them.

\begin{proposition} \label{GAR 1}
Assume that the following two conditions hold:

\vspace{0.5mm}

\noindent
{\rm{(1)}} $G/H$ is a Galois group over $k$,

\vspace{0.5mm}

\noindent
{\rm{(2)}} $H$ is $A_n$ with $n \not \in \{1,2,3,4,6\}$, or ${\rm{PSL}}_2(\mathbb{F}_p)$ with $p$ odd and $p \not = \pm1 \, \, {\rm{mod}} \, \, 24$, or any sporadic group with the possible exception of $M_{23}$.

\vspace{0.5mm}

\noindent
Then, condition {\rm{($*$)}} holds.
\end{proposition}

\begin{proof}
Assume that $H$ is as in condition (2). Then, by \cite[Chapter IV, Theorem 4.3]{MM99} (see also \cite[Remark 16.9.5]{FJ08}), $H$ has a so-called {\it{GAR-realization over $k$}}. Let $F/k$ be a Galois extension with Galois group $G/H$ (condition (1)), given with an isomorphism $\theta: G/H \rightarrow {\rm{Gal}}(F/k).$ Denote the canonical surjection of $G$ onto $G/H$ by $\varphi$. Then, the embedding problem $\theta \circ \varphi : G \rightarrow {\rm{Gal}}(L/k)$ has kernel $H$. A result of Matzat (see, e.g., \cite[Proposition 16.8.6]{FJ08}) states that every embedding problem whose kernel is non-abelian simple and possesses a GAR-realization over $k$ is regularly solvable over $k$. Since $k$ is Hilbertian, the regular solvability conclusion implies that this embedding problem has infinitely many solutions; cf., e.g., \cite[Lemma 16.4.2]{FJ08}. In particular, condition ($*$) holds.
\end{proof}

\subsubsection{A criterion for index 2 subgroups}

In Proposition \ref{Debes 92} below, we give a criterion which uses some regular realizations of $G$ over $k$.

\begin{proposition} \label{Debes 92}
Assume that the following two conditions hold.

\vspace{0.5mm}

\noindent
{\rm{(1)}} The quotient $G/H$ has order 2.

\vspace{0.5mm}

\noindent
{\rm{(2)}} There exists a $k$-regular Galois extension $E/k(T)$ with Galois group $G$ that satisfies the following two conditions:

\vspace{0.5mm}

{\rm{(a)}} $E/k(T)$ has genus at least 2,

\vspace{0.5mm}

{\rm{(b)}} the exact number of conjugacy classes of elements of even order 

belonging to the inertia canonical invariant of $E/k(T)$ is either 2 or 

3, and the associated two or three branch points are all $k$-rational.

\vspace{0.5mm}

\noindent
Then, condition {\rm{($*$)}} holds.
\end{proposition}

\begin{proof}
By condition (1), the subextension $E^{H}/k(T)$ has degree 2. Moreover, by part (b) of condition (2), it has exactly two branch points $t_1$ and $t_2$, and $t_1$ and $t_2$ are $k$-rational. Up to applying a change of variable, we may and will assume that $E^{H}=k(\sqrt{T})$. Let $b$ be any element of $k \setminus \{0\}$ which is not a root of unity. By \cite[Corollary 1.7]{Deb92}, there exists $a \in k \setminus \{0\}$ such that, for each sufficiently large positive integer $m$, the specialization of $E/k(T)$ at $ab^m$ has Galois group $G$. Fix such an element $a$. For every $m$ as before, the specialization of $E^{H}/k(T)$ at $ab^m$ is quadratic and one has $(E^{H})_{ab^m}= k(\sqrt{ab})$ or $(E^{H})_{ab^m} = k(\sqrt{a})$. To conclude, it suffices to show that all the specializations of $E/k(T)$ at $ab^m$, $m \geq 1$, provide infinitely many distinct extensions of $k$. But this claim holds by part (a) of condition (2) and Proposition \ref{twisting}.
\end{proof}

\subsection{Solving every embedding problem at least once}

Now, we give sufficient conditions for the following condition to hold:

\vspace{1mm}

\noindent
{\rm{($**$)}} {\it{Every Galois extension of $k$ with Galois group $G/H$ embeds into at least one Galois extension of $k$ with Galois group $G$.}}

\vspace{1mm}

\noindent
Condition {\rm{($**$)}} may fail (e.g., if $H=\Zz/2\Zz$, $G=\Zz/4\Zz$, and $k=\Qq$; see \cite[Theorem 1.2.4]{Ser92}). However, it holds in the following cases.

\begin{proposition} \label{GAR 2}
Condition {\rm{($**$)}} holds in the following two situations:

\vspace{0.4mm}

\noindent
{\rm{(1)}} $H$ is $A_n$ with $n \not \in \{1,2,3,4,6\}$, or ${\rm{PSL}}_2(\mathbb{F}_p)$ with $p$ odd and $p \not = \pm1 \, \, {\rm{mod}} \, \, 24$, or any sporadic group with the possible exception of $M_{23}$.

\vspace{0.4mm}

\noindent
{\rm{(2)}} $H=A_n$ and $G=S_n$ (for an arbitrary positive integer $n$).
\end{proposition}

\begin{proof}
In case (1), see the proof of Proposition \ref{GAR 1} to get condition {\rm{($**$)}}. In case (2), it is a result of Neumann; see \cite[Theorem 2]{Neu86}. 
\end{proof}

\section{Examples}

Let $k$ be a number field and $G$ a non-trivial finite group. This section is organized as follows. In \S5.1, we give examples of finite groups with no finite 1-parametric set over $k$; see Theorems 5.1-2. Then, in \S5.2, we give further examples of groups that have no parametric extension over $k$; see Theorem \ref{thm 2}. Finally, in \S5.3, we give conjectural examples of (infinite) sets of regular realizations over $k$ with the same Galois group which are not $n$-parametric over $k$ for large $n$; see Theorems 5.4-5. 

\subsection{Finite groups with no finite 1-parametric set over $k$}

First, we give examples over the given number field $k$.

\begin{theorem} \label{thm 1}
Assume that one of the following conditions holds:

\vspace{0.4mm}

\noindent
{\rm{(1)}} $G=G_1 \times G_2$, where $G_1$ is any finite group with $m_{G_1,k} \geq 2$ and $G_2$ is any non-trivial finite group,

\vspace{0.5mm}

\noindent
{\rm{(2)}} the order of $G$ has no prime factor $p$ such that $[k(\zeta_p):k] \leq 2$ (with $\zeta_p$ a primitive $p$-th root of unity)\footnote{See \S6.2.2 for more on these prime numbers.} and $G$ is not cyclic of prime order,

\vspace{0.4mm}

\noindent
{\rm{(3)}} $G$ is abelian, but none of the following groups:

\vspace{0.5mm}

{\rm{(a)}} $\Zz/n\Zz$ ($n \geq2$),

\vspace{0.1mm}

{\rm{(b)}} $(\Zz/p\Zz)^2$ ($p$ prime),

\vspace{0.5mm}

{\rm{(c)}} $\Zz/2\Zz \times \Zz/2p\Zz$ ($p$ prime),

\vspace{0.5mm}

{\rm{(d)}} $\Zz/3\Zz \times \Zz/3p\Zz$ ($p$ prime),

\vspace{0.5mm}

{\rm{(e)}} {\hbox{$\Zz/2\Zz \times \Zz/8\Zz$, $(\Zz/4\Zz)^2$, $(\Zz/2\Zz)^3$, $(\Zz/2\Zz)^2 \times \Zz/4\Zz$, $(\Zz/3\Zz)^3$, $(\Zz/2\Zz)^4$.}}

\vspace{0.5mm}

\noindent
{\rm{(4)}} the center $Z(G)$ of $G$ is not trivial and $G/Z(G)$ is neither solvable nor $A_5$ (for example, $G={\rm{GL}}_n(\mathbb{F}_{q})$ where $n$ is any integer and $q$ is any prime power such that $n \geq 2$, $q \geq 3$, and $(n,q) \not \in \{(2,3), (2,4)\}$).

\vspace{0.5mm}

\noindent
Then, $G$ has no non-empty finite 1-parametric set over $k$ \footnote{\label{foot1}See Remark \ref{rem1} for the case of the empty set.}.
\end{theorem} 

Now, we give other examples in the specific case $k=\Qq$ (in addition to those already given in Theorem \ref{thm 1}).

\begin{theorem} \label{thm 1.1}
Assume that one of the following conditions holds:

\vspace{0.5mm}

\noindent
{\rm{(1)}} $G$ is any non-abelian group such that $2$ does not divide $|G|$, $3$ divides $|G|$, and $G \not= (\Zz/p\Zz)^k \rtimes \Zz/3\Zz$ ($k \in \{1,2\}$, $p$ prime, $p \not=3$),

\vspace{0.5mm}

\noindent
{\rm{(2)}} $G$ is abelian, but none of the following groups:

\vspace{0.5mm}

{\rm{(a)}} $\Zz/p\Zz$ ($p$ prime),

\vspace{0.5mm}

{\rm{(b)}} $\Zz/4\Zz$, $\Zz/6\Zz$, $\Zz/12\Zz$, $(\Zz/2\Zz)^2$, $\Zz/2\Zz \times \Zz/4\Zz$, $\Zz/2\Zz \times \Zz/6\Zz$,

$(\Zz/3\Zz)^2$, $(\Zz/2\Zz)^3$, $(\Zz/2\Zz)^4$,

\vspace{0.5mm}

\noindent
{\rm{(3)}} $G$ is the dihedral group $D_n$ (with $2n$ elements), where $n$ is any positive integer which is neither a prime nor in $\{1,4,6,8,9,12\}$,

\vspace{0.5mm}

\noindent
{\rm{(4)}} the center $Z(G)$ of $G$ is not trivial and $G/Z(G)$ is neither solvable of even order nor of order $\leq 3$ (for example, $G={\rm{GL}}_2(\mathbb{F}_4)$).

\vspace{0.5mm}

\noindent
Then, $G$ has no non-empty finite 1-parametric set over $\Qq$ \footref{foot1}.
\end{theorem}

\subsection{Finite groups with no parametric extension over $k$} Next, we give other examples of finite groups with no parametric extension over $k$ (in addition to those already given in Theorems \ref{thm 1} and \ref{thm 1.1}).

\begin{theorem} \label{thm 2}
Assume that one of the following conditions holds:

\vspace{0.5mm}

\noindent
{\rm{(1)}} $G$ is one of the following abelian groups: $\Zz/6\Zz$, $\Zz/12\Zz$, $\Zz/2\Zz \times \Zz/4\Zz$, $\Zz/2\Zz \times \Zz/6\Zz$, $(\Zz/3\Zz)^2$, $(\Zz/2\Zz)^3$, $(\Zz/2\Zz)^4$,

\vspace{0.5mm}

\noindent
{\rm{(2)}} $G=S_n$, where $n$ is any integer $\geq 6$.

\vspace{0.5mm}

\noindent
Then, $G$ has no parametric extension over $\Qq$. Furthermore, if $G=S_n$ and $n \geq 8$, then $G$ has no parametric extension over $k$.
\end{theorem}

\subsection{Conjectural non-parametric sets}

First, we give conjectural examples of (infinite) sets of regular realizations over the number field $k$ with the same Galois group which are not $n$-parametric over $k$ for large integers $n$.

\begin{theorem} \label{thm 3}
Assume that the following three conditions hold:

\vspace{0.5mm}

\noindent
{\rm{(1)}} the Uniformity Conjecture of \S2.3 holds,

\vspace{0.5mm}

\noindent
{\rm{(2)}} $G$ is any group as in Theorem \ref{thm 1},

\vspace{0.5mm}

\noindent
{\rm{(3)}} $G$ is a Galois group over $k$ \footnote{\label{foot2}If $G$ is not a Galois group over $k$, then the conclusion of the result fails trivially.}.

\vspace{0.5mm}

\noindent
Then, given an integer $r_0 \geq 1$, there exists an integer $n \geq 1$ (depending only on $r_0$, $k$, and $|G|$) such that the set which consists of all the $k$-regular Galois extensions of $k(T)$ with Galois group $G$ and at most $r_0$ branch points is not $n$-parametric over $k$.
\end{theorem}

Now, we give other examples in the specific case $k=\Qq$ (in addition to those already given in Theorem \ref{thm 3}).

\begin{theorem} \label{thm 3.1}
Assume that the following three conditions hold:

\vspace{0.5mm}

\noindent
{\rm{(1)}} the Uniformity Conjecture of \S2.3 holds,

\vspace{0.5mm}

\noindent
{\rm{(2)}} either $G$ is any group as in conditions {\rm{(1)}}, {\rm{(2)}}, {\rm{(4)}} of Theorem \ref{thm 1.1} or $G$ is the dihedral group $D_n$ (with 2n elements), where $n$ is any integer which is not a prime and which has a prime factor $\geq 11$,

\vspace{0.5mm}

\noindent
{\rm{(3)}} $G$ is a Galois group over $\Qq$ \footref{foot2}.

\vspace{0.5mm}

\noindent
Then, given an integer $r_0 \geq 1$, there exists an integer $n \geq 1$ (depending only on $r_0$ and $|G|$) such that the set which consists of all the $\Qq$-regular Galois extensions of $\Qq(T)$ with Galois group $G$ and at most $r_0$ branch points is not $n$-parametric over $\Qq$.
\end{theorem}

\section{Proofs of Theorems 5.1-5}

\subsection{Direct products}

Assume $G=G_1 \times G_2$, where $G_1$ is any finite group with $m_{G_1,k} \geq 2$ and $G_2$ is any non-trivial finite group. Below, we prove that the conclusions of Theorems \ref{thm 1} and \ref{thm 3} hold with $G$.

Obviously, we may and will assume that $G$ is a regular Galois group over $k$. Since $m_{G_1,k} \geq 2$, it suffices to find a Galois extension $F_1/k$ with Galois group $G_1$ which embeds into infinitely many Galois extensions of $k$ with Galois group $G$; see Theorem \ref{genus2}. Let $F_1/k$ be a Galois extension with Galois group $G_1$ and $E_2/k(T)$ a $k$-regular Galois extension with Galois group $G_2$ (such extensions exist as $G$ is a regular Galois group over $k$). Consider the extension $E_2F_1 / F_1(T)$. By Hilbert's irreducibility theorem, there exist infinitely many $t_0 \in k$ such that the specializations of $E_2F_1/F_1(T)$ at $t_0$ are distinct and all have Galois group $G_2$ (as $G_2$ is not trivial). For such a $t_0$, one has $(E_2 F_1)_{t_0} = (E_2)_{t_0} F_1$. Hence, $(E_2)_{t_0}/k$ has Galois group $G_2$ and $(E_2)_{t_0}$, $F_1$ are linearly disjoint over $k$. In particular, $(E_2)_{t_0} F_1/k$ has Galois group $G$, as needed.

\subsection{Groups whose order has only large prime factors}

This section is organized as follows. In \S6.2.1, we prove the conclusions of Theorems \ref{thm 1} and \ref{thm 3} for the group $G$, where $G$ is any group as in condition (2) of Theorem \ref{thm 1}. \S6.2.2 is devoted to more properties on the set of prime numbers that appear in condition (2) of Theorem \ref{thm 1}. In \S6.2.3, we prove the conclusions of Theorems \ref{thm 1.1} and \ref{thm 3.1} for the group $G$, where $G$ is any group as in condition (1) of Theorem \ref{thm 1.1}.

\subsubsection{Over the number field $k$} \label{6.2.1}

Denote the set of all prime numbers $p$ such that $[k(\zeta_p):k] \leq 2$ by $\mathcal{S}$. Assume that the order $|G|$ of $G$ has no prime factor in $\mathcal{S}$ and $|G|$ is neither 1 nor a prime number. As 2 is in $\mathcal{S}$, the group $G$ has odd order and, since it is not cyclic of prime order, it is not simple. Let $H$ be a non-trivial proper normal subgroup of $G$. By Theorem \ref{genus2} and Remark \ref{odd order}, it suffices to prove the inequality $m_{G/H,k} \geq 2$ to get the desired conclusions.

First, assume that there exists a $k$-regular Galois extension of $k(T)$ with Galois group $G/H$ and genus 1. By Proposition \ref{RH genus 1}, $G/H$ contains an element of order $n$ for some $n \in \{2,3,4,6\}$. But this last conclusion cannot happen since $|G/H|$ has no prime factor in $\mathcal{S}$ and $\{2,3\} \subset \mathcal{S}$.

Now, assume that there exists a $k$-regular Galois extension $E/k(T)$ with Galois group $G/H$ and genus 0. Since $|G/H|$ is coprime to 6, one has $G/H \cong \Zz/n\Zz$ for some integer $n$ which is coprime to 6 and $E/k(T)$ has exactly two branch points; see Proposition \ref{list genus 0}. Given a prime divisor $p$ of $n$, consider the subextension $E'/k(T)$ of $E/k(T)$ that has degree $p$. By the definition of $\mathcal{S}$, one has $[k(\zeta_p):k] \geq 3$. Then, by the Branch Cycle Lemma, $E'/k(T)$ has at least $3$ branch points, which cannot happen. Hence, one has $m_{G/H,k} \geq 2$, thus ending the proof.

\subsubsection{More on the set $\mathcal{S}$}

Below, we make the set $\mathcal{S}$ of \S\ref{6.2.1} more explicit, at the cost of making it larger. Denote the set of all prime numbers by $\mathcal{P}$. Given $p \in \mathcal{P}$, let $\zeta_p$ be a primitive $p$-th root of unity.

\begin{proposition} \label{S1}
{\rm{(1)}} One has $\mathcal{S} \subset \{p \in \mathcal{P} \, : \, p \leq 2[k:\Qq]+1\}$.

\vspace{0.5mm}

\noindent
{\rm{(2)}} One has $\mathcal{S} \subset \{2,3\} \cup \{p \in \mathcal{P} \setminus \{2,3\} \, : \, k \cap \Qq(\zeta_p) \not= \Qq\}$.

\vspace{0.5mm}

\noindent
{\rm{(3)}} One has $\mathcal{S} \subset \{2,3\} \cup \{p \in \mathcal{P} \setminus \{2,3\} \, : \,  p \, \,  ramifies \,\,  in \, \, k/\Qq\}$.
\end{proposition}

\begin{proof}
Combine the definition of $\mathcal{S}$ and the inequality $[k(\zeta_p):k] \geq (p-1)/{[k:\Qq]}$ ($p$ prime)
to get the first inclusion. For the second one, fix $p \in \mathcal{S} \setminus \{2,3\}$ and suppose $k \cap \Qq(\zeta_p) = \Qq$. Then, $[k(\zeta_p):k] = p-1$. As $p$ is in $\mathcal{S}$, one gets $p \leq 3$, which cannot happen. As to the third one, let $p \in \mathcal{S} \setminus \{2,3\}$. By (2), one has $k \cap \Qq(\zeta_p) \not= \Qq$. If $p$ does not ramify in $k/\Qq$, then it does not ramify in $(k \cap \Qq(\zeta_p)) / \Qq$ either. In particular, the ramification index of $p$ in $\Qq(\zeta_p)/\Qq$ is at most $[\Qq(\zeta_p): k \cap \Qq(\zeta_p)] \leq (p-1)/2$, which cannot happen.
\end{proof}

As already said, one always has $\{2,3\} \subset \mathcal{S}$. In Proposition \ref{S2} below, we give sufficient conditions on $k$ for the converse to hold.

\begin{proposition} \label{S2}
One has $\mathcal{S} = \{2,3\}$ in each of the following cases:

\vspace{0.5mm}

\noindent
{\rm{(1)}} $k=\Qq$,

\vspace{0.5mm}

\noindent
{\rm{(2)}} $k/\Qq$ ramifies at most at 2 and 3.

\vspace{0.5mm}

\noindent
{\rm{(3)}} $k/\Qq$ contains no non-trivial cyclic subextension.
\end{proposition}

\begin{proof}
If either condition (1) or condition (2) holds, we may apply Proposition \ref{S1} to get $\mathcal{S}=\{2,3\}$. Now, assume that $k/\Qq$ contains no non-trivial cyclic subextension. We then need the following easy lemma, whose proof is left to the reader.

\begin{lemma} \label{lemma easy}
Let $p$ be an odd prime. Then, $[k(\zeta_p):k] \leq 2$ if and only if $k$ contains the unique subfield $F$ of $\Qq(\zeta_p)$ such that $[\Qq(\zeta_p):F]=2$.
\end{lemma}

\noindent
Let $p \in \mathcal{S} \setminus \{2\}$. Then, by Lemma \ref{lemma easy}, the field $k$ contains a subfield $F$ such that $F/\Qq$ is cyclic of degree $(p-1)/2.$ If $p \geq 5$, then $F/\Qq$ has degree at least 2, which cannot happen. Hence, $p = 3$, as needed.
\end{proof}

\subsubsection{Over the rationals}

Here, we assume that $G$ is a non-abelian finite group such that 2 does not divide $|G|$, 3 divides $|G|$, and $G \not= (\Zz/p\Zz)^k \rtimes \Zz/3\Zz$ ($k \in \{1,2\}$, $p$ prime, $p \not=3$). Below, we show that the conclusions of Theorems \ref{thm 1.1} and \ref{thm 3.1} hold for the group $G$.

As the group $G$ has odd order, it suffices to find a non-trivial normal subgroup $H$ of $G$ such that $m_{G/H, \Qq} \geq 2$; see Theorem \ref{genus2} and Remark \ref{odd order}. By Proposition \ref{list genus leq 1}, it then suffices to find a non-trivial proper normal subgroup $H$ of $G$ such that $G/H \not \cong \Zz/3\Zz$. Consider the derived subgroup $D(G)$ of $G$. It is a non-trivial proper characteristic subgroup of $G$ (since $G$ is solvable but not abelian). If $G/D(G) \not \cong \Zz/3\Zz$, we are done. One may then and will assume that $G/D(G) \cong \Zz/3\Zz$. If $D(G)$ has a non-trivial proper characteristic subgroup $H_0$, then one has $G/H_0 \not \cong \Zz/3\Zz$ (since $G/D(G) \cong \Zz/3\Zz$) and we are done. One may then and will assume that $D(G)$ has no non-trivial proper characteristic subgroup. As $D(G)$ is also solvable (and non-trivial), one has $D(G)= (\Zz/p\Zz)^k$ for some prime number $p$ and some positive integer $k$; see, e.g., \cite[3.3.15]{Rob96}. First, assume that $p=3$. Since $G/D(G)$ has order 3, we get that $G$ is a 3-group. Set $|G|=3^l$, where $l$ is a positive integer. If $l \in \{1,2\}$, then $G$ is abelian, which cannot happen. One then has $l \geq 3$. Pick a normal subgroup $H$ of $G$ with order 3. Then, $G/H \not \cong \Zz/3\Zz$ (as $l \geq 3$) and we are done. Now, assume that $p \not=3$. As $G/D(G) \cong \Zz/3\Zz$, one may then apply the Schur-Zassenhaus theorem to get $G = (\Zz/p\Zz)^k \rtimes \Zz/3\Zz$. Let $f: \Zz/3\Zz \rightarrow {\rm{Aut}} ((\Zz/p\Zz)^k)$ be the morphism defining the semidirect product. Denote the elements of $\Zz/3\Zz$ by $\bar{0}$, $\bar{1}$, and $\bar{2}$. As $G$ is not abelian, $f(\bar{1})$ is not the trivial automorphism of $(\Zz/p\Zz)^k$. Hence, $f(\bar{1})$ has order 3. In particular, there exists a non-zero element $y$ of $(\Zz/p\Zz)^k$ such that $y + f(\bar{1})(y)+f(\bar{2})(y)=0$ in $(\Zz/p\Zz)^k$. Then, the subgroup $H$ of $G$ generated by $(y,\bar{0})$ and $(f(\bar{1})(y),\bar{0})$ is a non-trivial proper normal subgroup of $G$ and, as $H$ is abelian, it has order at most $p^2$. From our assumption on $G$, one has $k \geq 3$. Hence, $G/H \not \cong \Zz/3\Zz$, thus ending the proof.

\subsection{Abelian groups}

This section is organized as follows. In \S6.3.1, we prove the conclusions of Theorems \ref{thm 1} and \ref{thm 3} for the group $G$, where $G$ is any group as in condition (3) of Theorem \ref{thm 1}. In \S6.3.2, we prove the conclusions of Theorems \ref{thm 1.1} and \ref{thm 3.1} for the group $G$, where $G$ is any group as in condition (2) of Theorem \ref{thm 1.1}. In \S6.3, we prove the conclusion of Theorem \ref{thm 2} for the group $G$, where $G$ is any group as in condition (1) of Theorem \ref{thm 2}.

\subsubsection{Over the number field $k$}

Here, we assume that $G$ is any non-trivial finite abelian group, but none of the following groups:

\noindent
{\rm{(a)}} $\Zz/n\Zz$ ($n \geq 2$),

\noindent
{\rm{(b)}} $(\Zz/p\Zz)^2$ ($p$ prime),

\noindent
{\rm{(c)}} $\Zz/2\Zz \times \Zz/2p\Zz$ ($p$ prime),

\noindent
{\rm{(d)}} $\Zz/3\Zz \times \Zz/3p\Zz$ ($p$ prime),

\noindent
{\rm{(e)}} $\Zz/2\Zz \times \Zz/8\Zz$, $(\Zz/4\Zz)^2$, $(\Zz/2\Zz)^3$, $(\Zz/2\Zz)^2 \times \Zz/4\Zz$, $(\Zz/3\Zz)^3$, $(\Zz/2\Zz)^4$.

\noindent
Below, we show that the conclusions of Theorem \ref{thm 1} and Theorem \ref{thm 3} hold for the abelian group $G$. By Theorem \ref{genus2}, Remark \ref{odd order}, and Proposition \ref{list genus leq 1 abelian}, it suffices to show that $G$ has a non-trivial proper subgroup $H$ such that $G/H$ is none of the following groups: $\Zz/n\Zz \, (n \geq 2), (\Zz/2\Zz)^2, \Zz/2\Zz \times \Zz/4\Zz, (\Zz/3\Zz)^2,(\Zz/2\Zz)^3.$

Set $G=\Zz/d_1\Zz \times \cdots \times \Zz/d_m\Zz,$ where $m \geq 1$ and $d_1, \dots,d_m$ are integers $\geq 2$ such that $d_{i}$ divides $d_{i+1}$ for each $i \in \{1,\dots,m-1\}$. As $G$ is not cyclic (condition (a)), one has $m \geq 2$. We split the proof into three cases: $m\ge 4$, $m=3$, and $m=2$. For short, we will call a non-trivial proper quotient of $G$ ``suitable" if it is not in the above list.

Assume that $m\geq 4$. Then, $\Zz/d_2\Zz \times \cdots \times \Zz/d_{m}\Zz$ is a suitable quotient of $G$, unless $m=4$ and $d_1= \cdots =d_4=2$. But this last case cannot happen by condition (e).

Assume that $m=3$. First, assume that $d_3 > d_1$ and $d_1 \not=2$. Then, $(\Zz/d_1\Zz)^3$ is a suitable quotient of $G$. Second, assume that $d_3>d_1$ and $d_1=2$. Then, $\Zz/d_2\Zz\times \Zz/d_3\Zz$ is a non-trivial proper quotient of $G$ that is suitable, unless $d_2=2$ and $d_3=4.$ But this last case cannot happen by condition (e). Third, assume that $d_1=d_2=d_3$. Then, $(\Zz/d_3\Zz)^2$ is a non-trivial proper quotient of $G$ that is suitable, unless $d_3 =2$ or $d_3=3$. But none of these two cases can happen by condition (e).

Assume that $m=2$. First, assume that $d_2>d_1\geq 4$. Then, $(\Zz/d_1\Zz)^2$ is a suitable quotient of $G$. Second, assume that $d_2=d_1 \geq 4$. Then, $d_2$ is not a prime number (condition (b)). Set $d_2=r \cdot s$, where $r$ and $s$ are integers $\geq 2$. By condition (e), one has $d_2 \geq 6$. Then, $\Zz/r\Zz \times \Zz/d_2\Zz$ is a suitable quotient of $G$. Third, assume that $d_1 \leq 3$. Then, by conditions (b), (c), and (d), $d_2/d_1$ is neither $1$ nor a prime. Set $d_2=d_1 \cdot r\cdot s$, where $r$ and $s$ are integers $\geq 2$. Then, $\Zz/d_1\Zz \times \Zz/(d_1 \cdot r)\Zz$ is a non-trivial proper quotient of $G$ that is suitable, unless $d_1=2$ and $d_1 \cdot r=4$. This last case leads to $r=2$ and, without loss, to $s=2$ as well. Hence, $G=\Zz/2\Zz \times \Zz/8\Zz$, which cannot happen by condition (e).

\subsubsection{Over $\Qq$}

Here, we assume that $G$ is any non-trivial finite abelian group as in condition (2) of Theorem \ref{thm 1.1}. If $G$ is as in condition (3) of Theorem \ref{thm 1}, then the conclusions of Theorems \ref{thm 1.1} and \ref{thm 3.1} hold with $G$. One may then assume that $G$ is one of the following groups:

\noindent
{\rm{(a)}} $\Zz/n\Zz$ ($n$ is neither a prime number nor in $\{4,6,12\}$),

\noindent
{\rm{(b)}} $(\Zz/p\Zz)^2$ ($p$ prime, $p \geq5$),

\noindent
{\rm{(c)}} $\Zz/2\Zz \times \Zz/2p\Zz$ ($p$ prime, $p \geq 5$),

\noindent
{\rm{(d)}} $\Zz/3\Zz  \times \Zz/3p\Zz$ ($p$ prime),

\noindent
{\rm{(e)}} $\Zz/2\Zz \times \Zz/8\Zz$, $(\Zz/4\Zz)^2$, $(\Zz/2\Zz)^2 \times \Zz/4\Zz$,  $(\Zz/3\Zz)^3$.

First, assume that $G$ is none of the following groups:

\noindent
{\rm{(i)}} $\Zz/2^a\Zz$ ($a \geq3$),

\noindent
{\rm{(ii)}} $\Zz/3^b\Zz$ ($b \geq2$),

\noindent
{\rm{(iii)}} $\Zz/2^a3^b \Zz$ ($a \geq 1$, $b \geq 1$, $(a,b) \not \in \{(1,1),(2,1)\}$),

\noindent
{\rm{(iv)}} $\Zz/2\Zz \times \Zz/8\Zz$, $(\Zz/4\Zz)^2$, $\Zz/3\Zz \times \Zz/6\Zz$, $\Zz/3\Zz \times \Zz/9\Zz$, $(\Zz/2\Zz)^2 \times \Zz/4\Zz$, $(\Zz/3\Zz)^3$.

\noindent
Then, there exist a prime $p \geq 5$ and a non-trivial proper subgroup $H$ of $G$ such that $G/H \cong \Zz/p\Zz$. Hence, $m_{G/H,\Qq} \geq 2$ by Proposition \ref{list genus leq 1 abelian}. Theorem \ref{genus2} and Remark \ref{odd order} then provide the desired conclusions.

Now, assume that $G$ is any group as in conditions (i)-(iv) above. Then, $G$ has a non-trivial proper subgroup $H$ such that $G/H$ is $\Zz/8\Zz$ or $\Zz/9\Zz$ or $\Zz/2\Zz \times \Zz/4\Zz$ or $ \Zz/3\Zz \times \Zz/3\Zz,$ unless $G=\Zz/8\Zz$ or $G=\Zz/9\Zz$. Hence, one has $m_{G/H,\Qq} \geq 2$; see Proposition \ref{list genus leq 1 abelian}. As above, one shows that the conclusions of Theorems \ref{thm 1.1} and \ref{thm 3.1} hold for the group $G$. 

Finally, assume that $G=\Zz/8\Zz$ (the other case for which $G=\Zz/9\Zz$ is similar). In this context, we refer to Lemma \ref{lemma genus2}. By Remark \ref{odd order}, it suffices to show that, given a $\Qq$-regular Galois extension $E/\Qq(T)$ with Galois group $\Zz/8\Zz$, the subextension $E^{\Zz/2\Zz}/\Qq(T)$ has genus at least 2. Denote the Euler function by $\varphi$. By the Branch Cycle Lemma, $E/\Qq(T)$ has at least $\varphi(8)=4$ branch points that are totally ramified. In particular, $E^{\Zz/2\Zz}/\Qq(T)$ has at least 4 branch points. Then, the genus of $E^{\Zz/2\Zz}/\Qq(T)$ is at least 2 by the Riemann-Hurwitz formula.

\subsubsection{Remaining cases}

Below, we assume that $G$ is among $\Zz/6\Zz$, $\Zz/12\Zz$, $\Zz/2\Zz \times \Zz/4\Zz$, $\Zz/2\Zz \times \Zz/6\Zz$, $(\Zz/3\Zz)^2$, $(\Zz/2\Zz)^3$, $(\Zz/2\Zz)^4$.

We need the following lemma, which is a standard consequence of the rigidity method; see, e.g., \cite[\S3.2]{Vol96}.

\begin{lemma} \label{general abelian}
Let $G_0$ be a non-trivial finite abelian group. Then, there exists a $\Qq$-regular Galois extension of $\Qq(T)$ with Galois group $G_0$ and whose inertia canonical invariant contains the conjugacy class $\{g_0\}$ for every non-zero element $g_0$ of $G_0$.
\end{lemma}

First, assume that $G=(\Zz/2\Zz)^3$ or $G=(\Zz/2\Zz)^4$. Then, $G$ contains at least 6 elements of order 2, which then generate maximal cyclic subgroups of $G$. One may then apply Theorem \ref{genus0} (with $H$ any subgroup of $G$ with order 2), as well as Remark \ref{odd order} and Lemma \ref{general abelian}, to get the conclusion of Theorem \ref{thm 2} for the group $G$.

Second, assume that $G$ is $\Zz/12\Zz$ or $\Zz/2\Zz \times \Zz/4\Zz$ or $\Zz/2\Zz \times \Zz/6\Zz$ or $(\Zz/3\Zz)^2$. By Theorem \ref{cyclic}, Remark \ref{odd order}, and Lemma \ref{general abelian}, it suffices to find a non-trivial subgroup $H$ of $G$ such that

\noindent
- $G/H$ is cyclic of order at least 3,

\noindent
- $n \geq 2+ |H|$, with $n$ the number of elements of $G$ with maximal order.

\noindent
But this last claim holds:

\noindent 
- if $G=\Zz/12\Zz$, one has $n=4$ and one can take $H=\Zz/2\Zz$,

\noindent
- if $G=\Zz/2\Zz \times \Zz/4\Zz$, one has $n=4$ and one can take $H=\Zz/2\Zz \times \{0\}$,

\noindent
- if $G=\Zz/2\Zz \times \Zz/6\Zz$, one has $n =4$ and one can take $H=\Zz/2\Zz \times \{0\}$,

\noindent
- if $G=(\Zz/3\Zz)^2$, one has $n=8$ and one can take $H=\Zz/3\Zz \times \{0\}$.

Third, assume that $G=\Zz/6\Zz$. Here, we apply Theorem \ref{cyclic2}. First, note that conditions (1), (2), and (3) are trivially satisfied. Moreover, condition (4) holds with the two conjugacy classes of the two elements of $G$ of order 6.

\subsection{Dihedral groups}

Let $n$ be a positive integer. Assume that $G$ is the dihedral group $D_n$ (with $2n$ elements).

\subsubsection{Proof of the conclusion of Theorem \ref{thm 3.1}}

First, assume that $n$ is not a prime number and $n$ has a prime factor $p \geq 11$. By \cite[Theorem 5.1]{DF94} and the Riemann-Hurwitz formula, $D_p  \cong D_n/(\Zz/(n/p)\Zz)$ satisfies $m_{D_p,\Qq} \geq 2$. It then remains to apply Theorem \ref{genus2} and Remark \ref{odd order} to get the conclusion of Theorem \ref{thm 3.1} for the group $G$ \footnote{This argument also provides the conclusion of Theorem \ref{thm 1.1} for the dihedral group $D_n$, under the extra assumption that $n$ has a prime factor $p \geq 11$.}.

\subsubsection{Proof of the conclusion of Theorem \ref{thm 1.1}}

Now, assume that $n$ is neither a prime number nor in $\{1,4,6,8,9,12\}.$ Let $p$ be the smallest prime factor of $n$. Then, $G$ has a unique normal subgroup $H$ of order $p$ and the quotient $G/H$ is isomorphic to $D_{n/p}$.

We then need the following two lemmas.

\begin{lemma} \label{lemma 2 dihedral}
Let $\mathcal{S}$ be a finite set of primes $q$ such that $q \equiv 1 \, \, {\rm{mod}} \, \, n$. Then, there exists a Galois extension $F/\Qq$ that satisfies the following two conditions:

\vspace{0.5mm}

\noindent
{\rm{(1)}} ${\rm{Gal}}(F/\Qq)=D_{n}$,

\vspace{0.5mm}

\noindent
{\rm{(2)}} the ramification index of $F/\Qq$ at $q$ is equal to $n$ for each $q \in \mathcal{S}$.
\end{lemma}

\begin{proof}
Given a prime number $q \in \mathcal{S}$, let $F_q/\Qq_q$ be a totally ramified cyclic extension of degree $n$ \footnote{For example, one can take $F_q/\Qq_q= \Qq_q(\sqrt[n]{q})/\Qq_q$ (as $q \equiv 1 \, \, {\rm{mod}} \, \, n$).}. Then, by \cite[Theorem 1.1]{DLN17}, there exists a Galois extension $F/\Qq$ such that ${\rm{Gal}}(F/\Qq)=D_{n}$ and the completion of $F/\Qq$ at $q$ is equal to $F_q/\Qq_q$ for each $q \in \mathcal{S}$. Hence, the lemma holds.
\end{proof}

\begin{lemma} \label{lemma 1 dihedral}
One has $m_{D_{n/p},\Qq} \geq 1$.
\end{lemma}

\begin{proof}
If $m_{D_{n/p},\Qq} =0$, then $n/p \in \{1,2,3,4,6\}$ by Proposition \ref{list genus 0}, i.e., $n$ is a prime number or $n \in \{4, 6,8,9,12\}$, which cannot happen.
\end{proof}

Given a prime $q$ such that $q \equiv 1 \, \, {\rm{mod}} \, \, n$, let $F/\Qq$ be a Galois extension with Galois group $D_n$ and ramification index $n$ at $q$ (Lemma \ref{lemma 2 dihedral}). Consider the subextension $F^H/\Qq$ (that has Galois group $D_{n/p}$). Since the ramification index of $q$ in $F/\Qq$ is equal to $n$, that in $F^H/\Qq$ is at least $n/p$ and, as $n$ is neither a prime number nor equal to 4, one has $n/p \geq 3$. Moreover, as $H$ is the unique normal subgroup of $D_n$ with order $p$, the extension $F^H/\Qq$ embeds into infinitely many Galois extensions of $\Qq$ with Galois group $D_n$; see Proposition \ref{solvable kernel}. It then remains to use Theorem \ref{genus1} and Lemma \ref{lemma 1 dihedral} to conclude.

\subsection{Symmetric groups over $k$}

Assume that $G=S_n$ for some $n \geq 8$. Below, we show that $G$ has no parametric extension over $k$. To do this, we apply Theorem \ref{genus0} (with $H=A_n$).  Recall that a permutation $\sigma \in S_n$ has {\it{type $1^{l_1} \dots n^{l_n}$}} if, for each index $i \in \{1,\dots, n\}$, there are $l_i$ disjoint cycles of length $i$ in the cycle decomposition of $\sigma$ (for example, an $n$-cycle has type $n^1$). Denote the conjugacy class in $S_n$ of elements of type $1^{l_1} \dots n^{l_n}$ by $[1^{l_1} \dots n^{l_n}]$.

\subsubsection{Checking condition {\rm{(1)}}}

Let $E/\Qq(T)$ be a $\Qq$-regular Galois extension with group $S_n$ and inertia canonical invariant $([1^{n-2}2^1],$  $[1^1(n-1)^1], [n^1])$ (such an extension exists; see, e.g., \cite[\S8.3.1]{Ser92}). By the Branch Cycle Lemma, the associated branch points all are $\Qq$-rational and this inertia canonical invariant contains exactly two conjugacy classes of elements of $S_n$ with even order. Moreover, by the Riemann-Hurwitz formula, the genus of $E/\Qq(T)$ is at least 2 (as $n \geq 5$). We may then apply Proposition \ref{Debes 92} to get condition (1) of Theorem \ref{genus0}.

\subsubsection{Checking condition {\rm{(2)}}}

First, we need Lemma \ref{symmetric} below, which relates to part (d) of condition (2) of Theorem \ref{genus0}.

\begin{lemma} \label{symmetric}
Every conjugacy class of $S_n$ which is not contained in $A_n$ belongs to the inertia canonical invariant of some $\Qq$-regular Galois extension of $\Qq(T)$ with Galois group $S_n$.
\end{lemma}

\begin{proof}
Let $C$ be a conjugacy class of $S_n$ which is not contained in $A_n$. Denote the type of all the elements of $C$ by $e_1^1 \dots e_s^1$. Let $T_1,\dots,T_n,T, Y$ be indeterminates that are algebraically independent over $\Qq$. Set
$$H(Y)=(Y-1)^{e_1} \cdots (Y-s)^{e_s} \in \Qq[Y],$$
$$G(Y)=Y^n + T_1 Y^{n-1} + \cdots + T_{n-1} Y + T_n \in \Qq(T_1,\dots,T_n)[Y]$$
and
$$F(Y)=G(Y) - T H(Y)  \in \Qq(T_1,\dots,T_n,T)[Y].$$
Clearly, $F$ has group $S_n$ over $\Qq(T_1,\dots,T_n,T)$. By Hilbert's irreducibility theorem, there are infinitely many $(t_1,\dots,t_n) \in \Qq^n$ such that
$$\overline{F}(Y):=Y^n + t_1 Y^{n-1} + \cdots + t_{n-1} Y + t_n - TH(Y) \in \Qq(T)[Y]$$
has Galois group $S_n$ over $\Qq(T)$. Pick such a tuple $(t_1,\dots,t_n)$ and denote the splitting field of $\overline{F}(Y)$ over $\Qq(T)$ by $E$. By \cite[Lemma 3.1]{Mue02}, $\infty$ is a branch point of $E/\Qq(T)$ and the associated inertia canonical conjugacy class is $C$. Then, the Galois group of $E\overline{\Qq}/\overline{\Qq}(T)$, which is a normal subgroup of $S_n$, contains an element of $S_n \setminus A_n$. Hence, ${\rm{Gal}}(E\overline{\Qq}/\overline{\Qq}(T))=S_n$, i.e., $E/\Qq$ is regular, thus ending the proof.
\end{proof}

Now, we prove condition {\rm{(2)}} of Theorem \ref{genus0}. By Lemma \ref{symmetric} and as each conjugacy class of $S_n$ is rational\footnote{Recall that a conjugacy class $C$ of a given finite group $G$ is {\it{rational}} if $C^i=C$ for each integer $i$ which is coprime to the order of the elements of $C$.}, it suffices to find 5 distinct conjugacy classes $C_1, \dots, C_5$ of $S_n$ such that $C_i \not \subset A_n$ and each element of $C_i$ generates a maximal cyclic subgroup of $S_n$ ($i \in \{1,\dots,5\}$). First, assume that $n \geq 11$. Then, for each integer $1 \leq m \leq n-1$ (resp., $1 \leq m \leq n-2$), the conjugacy class $[m^1(n-m)^1]$ (resp., $[1^1m^1(n-m-1)^1]$) has the desired properties if $n$ is odd (resp., if $n$ is even) and one has $(n-1)/2 \geq 5$ (resp., $(n-2)/2 \geq 5$) such conjugacy classes. In the remaining cases, these conjugacy classes have the desired properties:

\noindent
- $[10^1], [1^2 8^1], [1^1 2^1 7^1], [1^1 3^1 6^1], [1^1 4^1 5^1]$ (if $n=10$),

\noindent
- $[1^1 8^1], [2^1 7^1], [3^1 6^1], [4^1 5^1], [1^1 2^1 3^2]$ (if $n=9$),

\noindent
- $[8^1], [1^2 6^1], [1^1 2^1 5^1], [1^1 3^1 4^1], [2^13^2]$ (if $n=8$).

\vspace{-0.2mm}

\subsection{Symmetric groups over $\Qq$}

Assume that $G=S_n$ for some $n \geq 6$. Below, we show that $G$ has no parametric extension over $\Qq$. By \S6.5, one may and will assume that $n \in \{6,7\}$. In this case, we apply Theorem \ref{cyclic2}. First, note that condition (1) clearly holds, condition (2) holds by Proposition \ref{GAR 2}, and condition (3) holds as well (as explained in \S6.5.1). By Lemma \ref{symmetric}, it remains to find 3 distinct conjugacy classes $C_1,C_2,C_3$ of $S_n$ such that $C_i$ is not contained in $A_n$ and each element of $C_i$ generates a maximal cyclic subgroup of $S_n$ ($i \in \{1,2,3\}$). For $n=6$, one can take $\{C_1,C_2,C_3\} = \{[6^1], [1^2 4^1], [1^12^13^1]\}$. For $n=7$, one can take $\{C_1,C_2,C_3\}=\{[1^16^1], [2^1 5^1], [3^1 4^1]\}$, thus ending the proof.

\subsection{Linear groups (and more general groups)}

Here, we assume that the center $Z(G)$ of $G$ is not trivial and the quotient $G/Z(G)$ is neither solvable nor $A_5$. Below, we show that the conclusions of Theorems \ref{thm 1} and \ref{thm 3} hold with $G$. Without loss, we may and will assume that  $G$ is a regular Galois group over $k$. By Proposition \ref{solvable kernel}, there exists a Galois extension of $k$ with Galois group $G/Z(G)$ which embeds into infinitely many Galois extensions of $k$ with group $G$. Moreover, one has $m_{G/Z(G),k} \geq 2$ by Proposition \ref{list genus leq 1}. It then remains to apply Theorem \ref{genus2} to conclude. The same argument provides the conclusions of Theorems \ref{thm 1.1} and \ref{thm 3.1} for the group $G$ if $Z(G)$ is not trivial and $G/Z(G)$ is neither solvable of even order nor of order $\leq 3$.

In Proposition \ref{linear groups} below, we give explicit examples of finite groups $G$ as above.

\begin{proposition} \label{linear groups}
Let $n \geq 2$ be an integer and $q \geq 3$ a prime power. Set $G={\rm{GL}}_n(\mathbb{F}_q)$. Then, the following three conclusions hold.

\vspace{0.5mm}

\noindent
{\rm{(1)}} The center $Z(G)$ of $G$ is not trivial.

\vspace{0.5mm}

\noindent
{\rm{(2)}} Assume that $(n,q) \not \in \{(2,3), (2,4)\}$. Then, $G/Z(G)$ is neither solvable nor $A_5$.

\vspace{0.5mm}

\noindent
{\rm{(3)}} Assume that $(n,q)  \not=(2,3)$. Then, $G/Z(G)$ is neither solvable of even order nor of order $\leq 3$.
\end{proposition}

\begin{proof}
Conclusion (1) follows from the assumption $q \geq 3$. For conclusion (2), assume that $(n,q) \not \in \{(2,3), (2,4)\}$. If $G/Z(G)={\rm{PGL}}_n(\mathbb{F}_q$) was solvable, then this would be also true for ${\rm{PSL}}_n(\mathbb{F}_q)$. By, e.g., \cite[Theorem 9.46]{Rot95}, one would get $(n,q) \in \{(2,2), (2,3)\}$, which cannot happen. Moreover, it is easily checked that $|{\rm{PGL}}_n(\mathbb{F}_q)| \not= 60$ (as $(n,q) \not=(2,4)$). Hence, ${\rm{PGL}}_n(\mathbb{F}_q) \not \cong A_5$ and conclusion (2) holds. Now, assume that $(n,q) \not= (2,3)$. Then, by the above, $G/Z(G)$ is not solvable. In particular, conclusion (3) holds.
\end{proof}

\section{Appendix A: groups of minimal genus 0 or 1}

For this section, let $k$ be a number field, $G$ a non-trivial finite group, and $E/k(T)$ a $k$-regular Galois extension with Galois group $G$ and inertia canonical invariant $(C_1,\dots,C_r)$. For each $i \in \{1,\dots,r\}$, let $e_i$ be the order of the elements of $C_i$. The unordered $r$-tuple $(e_1,\dots,e_r)$ is called the {\it{ramification type of $E/k(T)$}}.

Below, we collect some well-known conclusions on the group $G$, the number of branch points $r$, and the ramification type $(e_1,\dots,e_r)$, under the assumption that the genus $g$ of $E/k(T)$ is either 0 or 1.

First, we consider the cases $g=0$ and $g=1$ separately. Propositions 7.1-2  below are essentially applications of the Riemann-Hurwitz formula and the Branch Cycle Lemma. See also, e.g., \cite[Chapter I, Theorem 6.2]{MM99}.

\begin{proposition} \label{list genus 0}
Assume that $g=0$. Then, the following two conclusions hold.

\vspace{0.5mm}

\noindent
{\rm{(1)}} One of the following conditions holds:

\vspace{0.5mm}

{\rm{(a)}} $G=\Zz/n\Zz$ for some $n \geq 2$, $r=2$, and $(e_1,e_2)=(n,n)$,

\vspace{0.5mm}

{\rm{(b)}} $G=D_n$ for some $n \geq 2$, $r=3$, and $(e_1,e_2,e_3)=(2,2,n)$,

\vspace{0.5mm}

{\rm{(c)}} $G=A_4$, $r=3$, and $(e_1,e_2,e_3)=(2,3,3)$,

\vspace{0.5mm}

{\rm{(d)}} $G=S_4$, $r=3$, and $(e_1,e_2,e_3)=(2,3,4)$,

\vspace{0.5mm}

{\rm{(e)}} $G=A_5$, $r=3$, and $(e_1,e_2,e_3) = (2,3,5)$.

\vspace{0.5mm}

\noindent
{\rm{(2)}} Assume that $k=\Qq$. Then, cases {\rm{(a)}} and {\rm{(b)}} cannot occur if $n  \not \in \{2,3,4,6\}$, and case {\rm{(e)}} does not occur either.
\end{proposition}

\begin{proposition} \label{RH genus 1}
Assume that $g=1$. Then, the following two conclusions hold.

\vspace{0.5mm}

\noindent
{\rm{(1)}} One of the following conditions holds:

\vspace{0.5mm}

{\rm{(a)}} $r=3$ and $(e_1,e_2,e_3)=(2,3,6)$,

\vspace{0.5mm}

{\rm{(b)}} $r=3$ and $(e_1,e_2,e_3)=(2,4,4)$,

\vspace{0.5mm}

{\rm{(c)}} $r=3$ and $(e_1,e_2,e_3)=(3,3,3)$,

\vspace{0.5mm}

{\rm{(d)}} $r=4$ and $(e_1,e_2,e_3,e_4)=(2,2,2,2)$.

\vspace{0.5mm}

\noindent
{\rm{(2)}} Assume that $k=\mathbb{Q}$. Then, cases {\rm{(1)}}, {\rm{(2)}}, and {\rm{(3)}} cannot occur.
\end{proposition}

Now, for the convenience of the reader, we summarize some important properties that hold in both the cases $g=0$ and $g=1$.

\begin{proposition} \label{list genus leq 1}
Assume that $g \leq 1$. Then, the following two conclusions hold.

\vspace{0.5mm}

\noindent
{\rm{(1)}} The group $G$ is either solvable or $A_5$.

\vspace{0.5mm}

\noindent
{\rm{(2)}} Assume that $k=\Qq$. Then, the group $G$ is either solvable of even order or $\Zz/3\Zz$. 
\end{proposition}

\begin{proof}
First, we show (1). As the desired conclusion is quite clear if $g=0$ (by Proposition \ref{list genus 0}), we may assume that $g=1$. Then, e.g., \cite[Proposition 2.4]{GT90} provides that $G$ is solvable. As for (2), it is a straightforward combination of Propositions 7.1-2 and the above conclusion (1). 
\end{proof}

Finally, we make the Galois group $G$ totally explicit, under the extra assumption that this group is abelian. As Proposition \ref{list genus leq 1 abelian} below is a straightforward application of Propositions 7.1-2 and the Branch Cycle Lemma, details are left to the reader.

\begin{proposition} \label{list genus leq 1 abelian}
Assume that $G$ is abelian and $g \leq 1$. Then, the following two conclusions hold.

\vspace{0.5mm}

\noindent
{\rm{(1)}} The group $G$ is one of the following finite groups: 
$\Zz/n\Zz \, (n \geq 2), (\Zz/2\Zz)^2, \Zz/2\Zz \times \Zz/4\Zz, (\Zz/3\Zz)^2, (\Zz/2\Zz)^3.$

\vspace{0.5mm}

\noindent
{\rm{(2)}} Assume that $k=\Qq$. Then, $G$ is one of the following finite groups:
$\Zz/2\Zz, \Zz/3\Zz, \Zz/4\Zz, \Zz/6\Zz, (\Zz/2\Zz)^2, (\Zz/2\Zz)^3.$
\end{proposition}

\section{Appendix B: on parametric quadratic extensions in low genus}

The aim of this section consists in proving Proposition \ref{r leq 4} on $\Qq$-regular quadratic extensions of $\Qq(T)$ below.

\begin{proposition} \label{r leq 4}
Let $E/\Qq(T)$ be a $\Qq$-regular quadratic extension with at most 4 branch points. Then, the following conditions are equivalent:

\vspace{0.5mm}

\noindent
{\rm{(1)}} $E/\Qq(T)$ is parametric over $\Qq$,

\vspace{0.5mm}

\noindent
{\rm{(2)}} $E/\Qq(T)$ has exactly two branch points and both are $\Qq$-rational.

\vspace{0.5mm}

\noindent
Moreover, if condition {\rm{(2)}} fails, then there exist infinitely many quadratic extensions of $\Qq$ each of which is not a specialization of $E/\Qq(T)$.
\end{proposition}

The proof is organized as follows. In \S8.1, we recall some standard background on hyperelliptic curves and their quadratic twists. In \S8.2, we explain how to translate the problem in terms of rational points on quadratic twists. Finally, we prove Proposition \ref{r leq 4} in \S8.3.

\subsection{Hyperelliptic curves and their quadratic twists} 

Let $P(T) \in \Zz[T]$ be a separable polynomial with degree $n$. Set $$P(T)=a_0 + a_1 T + \cdots + a_{n-1} T^{n-1} + a_n T^n.$$

\subsubsection{The even case} 

First, we assume that $n$ is even. Consider the equivalence relation $\sim$ on $\overline{\Qq}^3 \setminus \{(0,0,0)\}$ defined as follows:
$$(y_1,t_1,z_1) \sim (y_2,t_2,z_2)$$
if and only if there exists some $\lambda \in \overline{\Qq} \setminus \{0\}$ such that $$(y_2,t_2,z_2) = (\lambda^{n/2}y_1,\lambda t_1, \lambda z_1).$$ The quotient space $(\overline{\Qq}^3 \setminus \{(0,0,0)\} )/ \sim$ is a weighted projective space that is denoted by 
$$\mathbb{P}_{n/2,1,1}(\overline{\Qq}).$$ Given $(y,t,z) \in \overline{\Qq}^3 \setminus \{(0,0,0)\}$, the corresponding point in $\mathbb{P}_{n/2,1,1}(\overline{\Qq})$ is denoted by $[y:t:z].$ 

Set
$$P(T,Z)= a_0 Z^n + a_1 Z^{n-1} T + \cdots + a_{n-1} Z T^{n-1} +a_n T^n.$$
The equation $Y^2= P(T,Z)$ in $\mathbb{P}_{n/2,1,1}(\overline{\Qq})$ is {\it{the hyperelliptic curve associated with $P(T)$}}, denoted by $C_{P(T)}$. The set of all $\Qq$-rational points on $C_{P(T)}$, i.e., the set of all points $[y:t:z] \in \mathbb{P}_{n/2,1,1}(\overline{\Qq})$ such that $(y,t,z) \in \Qq^3 \setminus \{(0,0,0)\}$ and $y^2= P(t,z)$, is denoted by $C_{P(T)}(\Qq).$ A point $[y:t:z] \in C_{P(T)}(\Qq)$ is {\it{trivial}} if $y=0$ and {\it{non-trivial}} otherwise. Equivalently, $[y:t:z] \in C_{P(T)}(\Qq)$ is trivial if $z \not=0$ and $P(t/z)=0$.

Let $d$ be a squarefree integer. The hyperelliptic curve $Y^2 = d \cdot P(T,Z)$ associated with the polynomial $d \cdot P(T)$ is called {\it{the $d$-th quadratic twist of $C_{P(T)}$}}; we denote it by $C_{d \cdot P(T)}$. 

\subsubsection{The odd case}

The case where $n$ is odd is quite similar. For the convenience of the reader and to avoid confusion, we redetail it below.

Consider the following equivalence relation $\sim$ on $\overline{\Qq}^3 \setminus \{(0,0,0)\}$:
$$(y_1,t_1,z_1) \sim (y_2,t_2,z_2)$$
if and only if there exists some $\lambda \in \overline{\Qq} \setminus \{0\}$ such that $$(y_2,t_2,z_2) = (\lambda^{(n+1)/2}y_1,\lambda t_1, \lambda z_1).$$ The quotient space $(\overline{\Qq}^3 \setminus \{(0,0,0)\} )/ \sim$ is a weighted projective space that is denoted by $\mathbb{P}_{(n+1)/2,1,1}(\overline{\Qq}).$ Given $(y,t,z) \in \overline{\Qq}^3 \setminus \{(0,0,0)\}$, the corresponding point in $\mathbb{P}_{(n+1)/2,1,1}(\overline{\Qq})$ is denoted by $[y:t:z].$ 

Set
$$P(T,Z)= a_0 Z^{n+1} + a_1 Z^{n} T + \cdots + a_{n-1} Z^2 T^{n-1} +a_n Z T^n.$$
The equation $Y^2= P(T,Z)$ in $\mathbb{P}_{(n+1)/2,1,1}(\overline{\Qq})$ is {\it{the hyperelliptic curve associated with $P(T)$}}; we denote it by $C_{P(T)}$. The set of all $\Qq$-rational points on $C_{P(T)}$, i.e., the set of all points $[y:t:z] \in \mathbb{P}_{(n+1)/2,1,1}(\overline{\Qq})$ such that $(y,t,z) \in \Qq^3 \setminus \{(0,0,0)\}$ and $y^2= P(t,z)$, is denoted by $C_{P(T)}(\Qq).$ A point $[y:t:z] \in C_{P(T)}(\Qq)$ is {\it{trivial}} if $y=0$ and {\it{non-trivial}} otherwise. Equivalently, $[y:t:z] \in C_{P(T)}(\Qq)$ is trivial if either $z=0$ (this point, which is $[0:1:0]$, is the point at $\infty$) or $z \not=0$ and $t/z$ is a root of $P(T)$. 

Given a squarefree integer $d$, the hyperelliptic curve $Y^2 = d \cdot P(T,Z)$ associated with the polynomial $d \cdot P(T)$ is called {\it{the $d$-th quadratic twist of $C_{P(T)}$}}; we denote it by $C_{d \cdot P(T)}$. 

\subsection{From specializations of quadratic extensions to rational points on twisted hyperelliptic curves and vice-versa}

Let $P(T) \in \Zz[T]$ be a separable polynomial with degree $n$, roots $t_1,\dots,t_n$, and leading coefficient $a_n$. For short, we set $E=\Qq(T)(\sqrt{P(T)})$.

\vspace{2mm}

First, we make the set of branch points and the specializations of the $\Qq$-regular quadratic extension $E/\Qq(T)$ explicit.

\begin{lemma} \label{lemma 1}
The set of branch points of $E/\Qq(T)$ is either the set $\{t_1,\dots,t_n\}$ (if $n$ is even) or the set $\{t_1,\dots,t_n\} \cup \{ \infty \}$ (if $n$ is odd).
\end{lemma}

\begin{proof}
See, e.g., \cite[Proposition 6.2.3]{Sti09}.
\end{proof}

\begin{lemma} \label{lemma 2}
{\rm{(1)}} Let $t_0 \in \Qq \setminus \{t_1,\dots,t_n\}$. Then, $E_{t_0} = \Qq(\sqrt{P(t_0)}).$

\vspace{0.5mm}

\noindent
{\rm{(2)}} Assume that $n$ is even. Then, $E_\infty = \Qq(\sqrt{a_n})$.
\end{lemma}

\begin{proof}
First, we prove part (1). As $P(t_0) \not=0$, the polynomial $Y^2-P(t_0)$ is separable. Since $E$ is the splitting field of $Y^2-P(T)$ over $\Qq(T)$, $E_{t_0}$ is the splitting field of $Y^2-P(t_0)$ over $\Qq$, i.e., $E_{t_0} = \Qq(\sqrt{P(t_0)})$.

Now, we prove part (2). Set
$P(T)=a_0 + a_1 T + \cdots + a_{n-1}T^{n-1} + a_n T^n.$
One has
$$P(T)=(T^{n/2})^{2} \Big(\frac{a_0}{T^n}  + \frac{a_1}{T^{n-1}} + \cdots + \frac{a_{n-1}}{T} + a_n \Big).$$
Set $U=1/T$ and $Q(U)= a_n + a_{n-1} U + \cdots + a_1 U^{n-1} + a_0 U^n.$ Since $n$ is even, $E$ is the splitting field of $Y^2-Q(U)$ over $\Qq(U)$ and, as 0 is not a root of $Q(U)$, the polynomial $Y^2-Q(0)$ is separable. Hence, the field $E_\infty$ is equal to the splitting field of $Y^2-Q(0)$ over $\Qq$, i.e., $E_\infty=\Qq(\sqrt{a_n})$, thus ending the proof of the lemma.
\end{proof}

Now, we need the following bridge between specializations of the extension $E/\Qq(T)$ and rational points on quadratic twists of $C_{P(T)}$.

\begin{lemma} \label{twisting explicit}
Let $d$ be a squarefree integer. Then, the following two conditions are equivalent:

\vspace{0.5mm}

\noindent
{\rm{(1)}} the quadratic extension $\Qq(\sqrt{d})/\Qq$ occurs as a specialization of the $\Qq$-regular quadratic extension $\Qq(T)(\sqrt{P(T)})/\Qq(T)$,

\vspace{0.5mm}

\noindent
{\rm{(2)}} the $d$-th quadratic twist $C_{d \cdot P(T)}$ has a non-trivial $\Qq$-rational point.
\end{lemma}

\begin{proof}

First, assume that $\Qq(\sqrt{d})/\Qq$ occurs as the specialization of $E/\Qq(T)$ at $t_0$ for some $t_0 \in \Qq\setminus \{t_1,\dots,t_n\}$. By part (1) of Lemma \ref{lemma 2}, one then has $\Qq(\sqrt{d}) = \Qq(\sqrt{P(t_0)})$. We then get $y^2=d \cdot P(t_0)$ for some $y \in \Qq \setminus \{0\}$. Hence, the non-trivial $\Qq$-rational point $[y:t_0:1]$ lies on the quadratic twist $C_{d \cdot P(T)}$. Now, assume that $n$ is even and $\Qq(\sqrt{d})/\Qq$ occurs as the specialization of $E/\Qq(T)$ at $\infty$. By part (2) of Lemma \ref{lemma 2}, one has $\Qq(\sqrt{d}) = \Qq(\sqrt{a_n})$. We then get $y^2=d \cdot a_n$ for some $y \in \Qq \setminus \{0\}$, i.e., the non-trivial $\Qq$-rational point $[y:0:1]$ lies on the quadratic twist $C_{d \cdot P(T)}$. Hence, implication (1) $\Rightarrow$ (2) holds.

Conversely, assume that $C_{d \cdot P(T)}$ has a $\Qq$-rational point $[y:t:z]$ with $y \not=0$. First, assume that $n$ is odd. Then, $z \not=0$ and $t/z$ is not a root of $P(T)$, i.e., $t/z$ is not a branch point of $E/\Qq(T)$; see Lemma \ref{lemma 1}. From
the equality $y^2=d \cdot z^{n+1} \cdot P(t/z)$, part (1) of Lemma \ref{lemma 2}, and as $n$ is odd, we get $\Qq(\sqrt{d}) =E_{t/z}$, as needed for (1). Now, assume that $n$ is even. If $z=0$, then $y^2=d \cdot a_n t^n$. This gives $\Qq(\sqrt{d}) = \Qq(\sqrt{a_n})$ (as $n$ is even and $t \not=0$) and then $\Qq(\sqrt{d}) = E_\infty$ by part (2) of Lemma \ref{lemma 2}. If $z \not=0$, then $t/z$ is not a root of $P(T)$, i.e., $t/z$ is not a branch point of $E/\Qq(T)$. From the equality $y^2=d \cdot z^{n} \cdot P(t/z)$, part (1) of Lemma \ref{lemma 2}, and as $n$ is even, we get $\Qq(\sqrt{d}) =E_{t/z}$, thus ending the proof of the lemma.
\end{proof}

\subsection{Proof of Proposition \ref{r leq 4}}

Set $E=\Qq(T)(\sqrt{P(T)})$, where $P(T)$ $\in \Zz[T]$ is separable. Since $E/\Qq(T)$ has at most 4 branch points, $P(T)$ has degree at most 4 (Lemma \ref{lemma 1}). First, assume that $P(T)$ has degree 2. Then, the desired conclusion follows from \cite[Proposition 3.1 and \S3.1]{Leg15}. Now, assume that $P(T)$ has degree 4 and no root in $\Qq$. Then, by \cite[Theorem 3.1 and \S3.4.1]{Leg16b}, there exist infinitely many squarefree integers $d$ such that the $d$-th quadratic twist $C_{d \cdot P(T)}$ has no $\Qq$-rational point. It then remains to use Lemma \ref{twisting explicit} to get the desired conclusion. Finally, assume that either $P(T)$ has degree 3 or $P(T)$ has degree 4 and a root in $\Qq$. Then, $E/\Qq(T)$ has a $\Qq$-rational branch point by Lemma \ref{lemma 1}. Up to applying a change of variable, we may assume that this branch point is $\infty$, i.e., we may assume that $P(T)$ has degree 3; see Lemma \ref{lemma 1}. Moreover, we may assume that $P(T)=T^3+bT+c$, where $b$ and $c$ are in $\Zz$, i.e., $C_{P(T)}$ is elliptic. The desired conclusion then follows from the following classical two results (and Lemma \ref{twisting explicit}):

\noindent
- $C_{P(T)}$ has infinitely many quadratic twists with Mordell-Weil rank 0 (see, e.g., \cite{Dab08} for references),

\noindent
- there exist only finitely many squarefree integers $d$ such that $C_{d \cdot P(T)}$ has a non-trivial torsion point; see \cite[Proposition 1]{GM91}.

\begin{remark}
(1) As shown in \cite{DD09}, the conclusion of Proposition \ref{r leq 4} fails in general if $\Qq$ is replaced by a larger number number field $k$.

\vspace{1mm}

\noindent
(2) To our knowledge, whether a given $\Qq$-regular quadratic extension of $\Qq(T)$ with at least 6 branch points is parametric over $\Qq$ is an open problem in general. However,

\noindent
- \cite[Proposition 5.6]{Leg16b} shows that the proportion of $\Qq$-regular quadratic extensions of $\Qq(T)$ with given number of branch points, ``height" at most $H$, and that are not parametric tends to 1 as $H$ tends to $\infty$,

\noindent
- by \cite[Corollary 1 and Conjecture 1]{Gra07} and Lemma \ref{twisting explicit}, each $\Qq$-regular quadratic extension of $\Qq(T)$ with 6 branch points or more is not parametric over $\Qq$ under conjectures (for example, the {\it{abc}}-conjecture).
\end{remark}

\bibliography{Biblio2}
\bibliographystyle{alpha}

\end{document}